\documentclass[10pt]{amsart}

\usepackage{amssymb,amsmath}
\usepackage{amsthm}
\usepackage{esint}
\usepackage{xcolor}
\usepackage{comment}
\usepackage{mathrsfs}

\title{A density-constrained model for Chemotaxis }

\author[I. Kim]{Inwon Kim}
\address{Department of Mathematics, UCLA,  Los Angeles, CA} 
\email{ikim@math.ucla.edu}

\author[A.Mellet]{Antoine Mellet}
\address{Department of Mathematics, University of  Maryland, College Park, MD}
\email{mellet@umd.edu}

\author[Y. Wu]{Yijing Wu}
\address{Department of Mathematics, University of  Maryland, College Park, MD}
\email{yijingwu@umd.edu}

\thanks{I. Kim was partially supported by NSF Grant DMS-1900804.\\
A. Mellet was partially supported by NSF Grant DMS-2009236.
}

\def\R{\mathbb R}

\def\P{\mathcal P}
\def\brho{\bar \rho}
\def\vphi{\varphi}
\def\pa{\partial}
\def\na{\nabla}
\def\div{\mathrm{div}\,}
\def\supp{\mathrm{Supp}\,}

\def\H{\mathcal H}

\numberwithin{equation}{section}

\newtheorem{theorem}{Theorem}[section]
\newtheorem{theorem*}{Theorem}
\newtheorem{remark}[theorem]{Remark}
\newtheorem{lemma}[theorem]{Lemma}

\newtheorem{proposition}[theorem]{Proposition}
\newtheorem{corollary}[theorem]{Corollary}
\newtheorem{definition}[theorem]{Definition}

\begin{document}

\begin{abstract}
We consider a model of congestion dynamics with chemotaxis: The density of cells follows a chemical signal it generates, while subject to an incompressibility constraint. The incompressibility constraint results in the formation of {\it patches}, describing regions where the maximal density has been reached. The dynamics of these patches can be described by either Hele-Shaw or Richards equation type flow (depending on whether we consider the model with diffusion or the model with pure advection). 
Our focus in this paper is on the construction of weak solutions for this problem via a variational discrete time scheme of JKO type. We also establish the uniqueness of these solutions.
In addition, we make more rigorous the connection between this incompressible chemotaxis model and the free boundary problems describing the motion of the patches in terms of the density and associated pressure variable. In particular, we obtain new results characterizing the pressure variable as the solution of an obstacle problem and prove that in the pure advection case the dynamic preserves patches.


\end{abstract}

\maketitle

\section{Introduction}
\subsection{Motivations and Overview of the paper's objectives}
The classical parabolic-elliptic Patlak-Keller-Segel model for chemotaxis is:
$$
\begin{cases}
\pa_t \rho - \mu \Delta \rho +\chi \div(\rho \na \phi)=0 , & \mbox{ in } \Omega\times (0,\infty)\\
 \eta  \Delta \phi  +\theta \rho - \sigma\phi =0  & \mbox{ in } \Omega\times (0,\infty) \end{cases}
$$
where $\rho$ denotes the cell density and $\phi$ the concentration of some chemical. The parameters $\mu$ and $\eta$ are the cell and chemical diffusivity, $\chi$ is the cell sensitivity, and $\theta$ and $\sigma$ describe the production and degradation of the chemical (see e.g. \cite{Keller_Segel}, \cite{Patlak}, \cite{HP})

\medskip
In this model, the diffusion  competes with the aggregating potential $\phi$ and it is well known that such an aggregation-diffusion equation might lead to concentration and possibly finite time blow-up of the density, as the diffusion is not strong enough to balance the attractive potential (see e.g.  \cite{DP}, \cite{HV}).
In this paper we want to take into account the incompressibility of the cells in order to investigate the behavior of the density $\rho$ after saturation occurs.
One way to enforce a constraint $\rho\leq \rho_M $ is to consider the limit $m\to \infty$ of the nonlinear diffusion version of the PKS model:
\begin{equation}\label{eq:PME}
\begin{cases}
\pa_t \rho - \mu \Delta (\rho/\rho_M)^m + \chi \div(\rho \na \phi)=0 ,\\
 \eta  \Delta \phi  +\theta \rho - \sigma\phi =0 .
 \end{cases}
\end{equation}
This approach (stiff, or incompressible, limit of a porous media type equation) has been used in numerous papers, in particular in the context of tumor growth (\cite{PQV},\cite{MPQ}, \cite{KP}, \cite{DS}) and chematoxis (\cite{CKY}, \cite{HLP}). 
Note that in this limit, the very strong cell diffusion in the region $\{\rho>\rho_M\}$ enforces the constraint $\rho\leq \rho_M$, while the diffusion vanishes in the region $\{\rho<\rho_M\}$.

\medskip

In this paper, we directly address the incompressible problem, without the intermediary \eqref{eq:PME}, using an approach that has been successfully used, in particular, in the study of congested crowd motion (see \cite{MRS,S_survey,MeSan16}). Formally, the idea is to project the desired velocity $ v=- \mu \na  \log \rho +\chi \na \phi $ 
onto the set $C(\rho)$ of admissible velocities, which preserve the constraint $\rho\leq\rho_M$:
$$
\begin{cases}
\pa_t \rho + \div (\rho P_{C(\rho)}(- \mu \na  \log \rho +\chi \na \phi )  ) =0, \qquad\rho\leq \rho_M \\
 \sigma \phi -  \eta\Delta \phi = \rho.
\end{cases}
$$
Definitions of the admissible set $C(\rho)$ and the projection operator $P_{C(\rho)}$, which will not be used in our subsequent analysis, are recalled in Appendix \ref{sec:projection}.
Since the term $- \mu \na  \log \rho$ cannot increase the maximum value of $\rho$, we can also write 
$
P_{C(\rho)}(- \mu \na  \log \rho +\chi \na \phi )=  - \mu \na  \log \rho+ P_{C(\rho)}( \chi \na \phi )
$ (see Remark \ref{rmk:projk}).
For convenience, we replace $\rho$ with $\rho_M \rho$ and $\phi$ with $\chi \phi $ to arrive at the normalized version,
\begin{equation}\label{eq:transpcons}
\begin{cases}
\pa_t \rho -\mu \Delta \rho + \div(\rho P_{C(\rho)} (\na \phi))=0 , \qquad\rho\leq 1 \\
 \sigma \phi -  \eta\Delta \phi = \rho,
\end{cases}
\end{equation}
where we redefined the constants $\sigma$ and $\eta$ as necessary ($\eta / (\chi \rho_M \theta) \mapsto \eta$, $\sigma / (\chi \rho_M \theta) \mapsto \sigma$).
In the sequel, we will take $\sigma=1$ and $\eta=1$ since these constants do not play any role in our results.

\medskip

The goal of this paper is to construct weak solutions of \eqref{eq:transpcons} and study their properties (e.g. uniqueness and relation to free boundary problems) both when $\mu=0$ and when $\mu>0$. 
This analysis will lay the foundation for a subsequent paper in which we investigate the asymptotic behavior of the solutions when $\eta\ll1$ and $t\sim \eta^{-1}$.
Throughout our paper, $\Omega$ is a fixed bounded subset of $\R^d$ with smooth boundary and the equation will be supplemented by Neumann boundary conditions that guarantee conservation of mass for $\rho$, and general Robin boundary conditions for $\phi$ which account for possible degradation of the chemical along $\pa\Omega$ (these boundary conditions do not significantly affect the analysis developed in this paper, but play an important role in the singular limit $\eta\ll1$).

\medskip

A rigorous approach to \eqref{eq:transpcons} was developed in \cite{MRS} when $\mu = 0 $ and $\na \phi$ is a fixed velocity field by using variational methods (see also \cite{MeSan16} for the case $\mu>0$). The existence of a weak solution was proved via an appropriate JKO type scheme with constraint.
As we will see in this paper, we can  extend  this approach to our case. 
Following  \cite{MRS}, we will not  use the projection operator $P_{C(\rho)}$, but instead introduce the pressure $p$ such that (see Appendix \ref{sec:projection})
$$   P_{C(\rho)} (\na \phi) =  \na \phi  - \na p, \qquad p\geq 0, \quad p(1-\rho)=0 \mbox{ a.e.}.$$
While this pressure $p$ appears naturally as the limit of $p_m=\rho^{m-1}$ in the stiff limit of the porous media equation \eqref{eq:PME}, it appears, in this variational approach, as a Lagrange multiplier for the constraint $\rho\leq 1$ (like the pressure in incompressible fluid mechanics equations). It is thus not surprising that $p=0$ when $\rho\neq 1$, as guaranteed by the condition $ p(1-\rho)=0$ a.e.
The conditions on $p$ can also be expressed by writing $p\in P(\rho)$ with
\begin{equation}\label{eq:P} P(\rho):= \begin{cases}
0 & 0\leq \rho <1 \\
[0,\infty) &  \rho =1
\end{cases}
\end{equation}
which is sometimes referred to as the Hele-Shaw graph.
\medskip

Using the fact that $\rho \na p = \na p$, we obtain the following simple model for chemotaxis in incompressible settings:
\begin{equation}\label{eq:weakHS}
\begin{cases}
\pa_t \rho  -\mu \Delta \rho + \div( \rho \na \phi-\na p)=0 ,\qquad p\in P(\rho),\\
  \phi -  \Delta \phi = \rho .
 \end{cases}
\end{equation}
This equation is the main topic of this paper when $\mu\geq 0$. 
Of particular interest  is the evolution of the saturated region $\Omega_s(t) = \{ \rho(\cdot,t)=1\}$ and the free boundary $\Sigma(t) = \pa \Omega_s(t)$.
\medskip

When $\mu>0$, equation \eqref{eq:weakHS} is a weak formulation for the problem
$$
\begin{cases}
\rho\in(0,1), \quad p=0, \qquad \pa_t \rho -\mu \Delta \rho + \div(\rho\na \phi) =0  & \mbox{ in } \Omega\setminus \Omega_s(t) \\
\rho =1, \quad p>0, \qquad \Delta p =\Delta \phi & \mbox{ in } \Omega_s(t)
\end{cases}
$$
with the free boundary $\Sigma$ being determined by the conditions
$$ \rho=1, \quad p=0, \quad \mu \na \rho\cdot \nu|_{\Sigma_+} = \na p\cdot\nu|_{\Sigma_-}$$
where $\nu$ denotes the normal unit vector along $\Sigma$ (pointing outward of $\Omega_s$).
The last condition comes from the fact that $\Delta (\mu \rho+p)$ cannot have a singular part along $\Sigma$ in \eqref{eq:weakHS}. This problem can be seen as a two phase free boundary problem for the function $u=\mu(\rho-1)+p$, which solves:
\begin{equation}\label{eq:Rich}
\begin{cases}
 \pa_t u +\div(u\na \phi)   = \mu (\Delta u -\Delta\phi)   & \mbox{ in } \{u<0\}  \\
 \Delta u-\Delta \phi =0& \mbox{ in } \{u>0\}
\end{cases}
\end{equation}
with the zero jump conditions 
\begin{equation}\label{zero}
[u]= [\na u\cdot \nu]=0 \quad \hbox{ on } \pa\{u>0\}.
\end{equation}

 While the problem is similar to the {\it Richards equation} used in filtration models (see e.g. \cite{AL}), the discontinuity of the drift term across the interface makes it more singular than the standard problem.
 
   \medskip

When $\mu=0$, equation \eqref{eq:weakHS} is a weak formulation for the problem
\begin{equation}\label{eq:HS1}
\begin{cases}
\rho\in[0,1), \quad p=0, \qquad \pa_t \rho + \div(\rho\na \phi) =0  & \mbox{ in } \Omega\setminus \Omega_s(t) \\
\rho =1, \quad p>0, \qquad \Delta p =\Delta \phi & \mbox{ in } \Omega_s(t)
\end{cases}
\end{equation}
(the pressure $p$ is still continuous across $\Sigma$, but $\rho$ might have a jump) where
the free boundary now moves according to the velocity law
\begin{equation}\label{eq:HS2}
 (1-\rho|_{\Sigma_+}) V = (-\na p+\na\phi)\cdot \nu  |_{\Sigma_-}
 \end{equation}
($V$ denotes the normal velocity of $\Sigma(t)$).
When the density is a characteristic function of some set, that is $\rho (x,t)= \chi_{\Omega_s(t)}(x)$, we recognize the usual one phase Hele-Shaw problem without surface tension:
\begin{equation}\label{eq:introHS}
\begin{cases}
\Delta p = \Delta \phi \mbox{ in } \Omega_s(t), \qquad p=0 \mbox{ on } \pa \Omega_s(t)\\
V  = (-\na p+\na\phi)\cdot \nu   \quad \mbox{ on } \pa \Omega_s(t).
\end{cases}
\end{equation}

In other words, in this fully  saturated regime, the chemotaxis system \eqref{eq:weakHS} 
 can be seen as a free boundary problem describing the motion of the region occupied by the cell, driven by the chemical concentration $\phi$ and the pressure variable $p$. 
However, since we obtained \eqref{eq:weakHS} by imposing the constraint $\rho\leq1 $, but without requiring $\rho\in\{0,1\}$, it is not clear that we should actually have  $\rho (x,t)= \chi_{\Omega_s(t)}(x)$ in general.
Using the fact that the potential $\phi$ is attractive (more precisely, the fact that $-\Delta\phi =\eta^{-1}( \rho -   \phi) \geq 0$ in $\{\rho=1\}$), we will  show that 
 if $\rho$ is a characteristic function at $t=0$, then this remains true for all $t>0$.

\medskip

 The transition from \eqref{eq:Rich}-\eqref{zero} to \eqref{eq:introHS} as $\mu\to 0$ is interesting, particularly due to the emergence of the free boundary velocity law in the limit. This connection between the Richards-type equation and the Hele-Shaw flow has been observed earlier in \cite{KPS} when the drift term was replaced by a growth term. Note that the convergence of $u$ to $p$ is only of zeroth order, due to the discontinuity of $\rho$ in the limit $\mu=0$ across the interface. This partially explains the abrupt change of the interface conditions from \eqref{zero} to the velocity law. Our analysis yields that $\nabla p$ strongly converges in the limit $\mu \to 0$ (Proposition~\ref{prop:obs}), serving as the regular part in the limit of the flux $\nabla u$.

 \medskip

Below is a brief summary of the main results of this paper:
\begin{enumerate}
\item Adapting the variational framework of \cite{MRS,MeSan16}, we construct weak solutions of \eqref{eq:weakHS} as limit of a discrete-time variational JKO scheme (\cite{JKO})  in  a bounded domain.
\item We show that for all $t>0$ the pressure $p(t)$ can be defined as the unique solution of a simple obstacle problem. This result is similar to a result obtained in \cite{GKM} in the context of tumor growth (without the potential $\phi$). The derivation of this obstacle problem relies on the variational nature of the JKO scheme and is thus very different from the proof presented in \cite{GKM}, which involved the porous media type equation \eqref{eq:PME}. But it does not require any technical a priori estimates and it yields the complementarity condition:
$$ p(\Delta p -\Delta \phi)=0.$$
We  also show that the discrete time approximation of the pressure, given by the variational scheme, is a subsolution of the same obstacle problem for all $t>0$.
\item Using a duality method, as in \cite{PQV}, we prove that \eqref{eq:weakHS} has a weak unique solution and thus fully characterizes the limit of the JKO scheme.

\item When $\mu=0$ (no diffusion) we show that 
 if $\rho$ is  initially a characteristic function, then this remains true for all $t>0$ (such a property cannot hold when $\mu>0$ because of the regularizing nature of diffusion). 
A similar result, when $\phi$ is a fixed potential satisfying $-\Delta\phi \geq 0$, was proved in \cite{AKY} (see also \cite{CKY}) using viscosity solution type arguments, based on the comparison principle that holds in this case. For our system the comparison principle for densities no longer hold.
We develop a very different, measure theoretic, approach which is simple and fits well with our notion of solutions.

\end{enumerate}
\medskip
\medskip

In a companion paper \cite{KMW2}, we will investigate the singular limit 
for this problem when $\phi$ solves $\phi -  \eta\Delta \phi = \rho $ with $\eta\ll1$.
We  then  show that at an appropriate time scale, the evolution of the saturated patches is described by  the Hele-Shaw free boundary problem with surface tension.

\subsection{Setting and notations}
We recall that $\Omega$ is a bounded subset of $\R^d$ with smooth boundary. Throughout the paper we denote the space-time domain $\Omega_T:= \Omega \times  (0,T))$ for any given $T>0$.

\medskip

We  introduce
\begin{equation}\label{eq:K}
K:=\left\{ \rho\in \P(\Omega), \; \rho(x)\leq 1 \mbox{ a.e. in }\Omega\right\}
\end{equation} 
where $\P(\Omega)$ denotes the set of probability measures on $\Omega$. In particular all $\rho\in K$ are absolutely continuous with respect to the Lebesgue measure and we can identify the measure with its density.
The set $K$ is equipped with the usual Wasserstein distance, defined by
$$
W_2^2(\mu,\nu) = \inf_{\pi \in \Pi(\mu,\nu)}  \int_{\Omega\times\Omega} |x-y|^2 d\pi(x,y)
$$
where $\Pi(\mu,\nu)$ denotes the set of all probability measures $\pi \in \mathcal P(\Omega\times\Omega)$ with marginals $\mu$ and $\nu$.
Given $\rho\in K$, the solution $\phi$ of 
\begin{equation}\label{eq:phi0}
\begin{cases}
 \phi -\Delta  \phi = \rho  &  \mbox{ in } \Omega\\
 \alpha \phi + \beta \nabla \phi \cdot n = 0 & \mbox{ on } \pa\Omega
\end{cases}
\end{equation}
(with $\alpha\geq 0$,  $\beta \geq 0$ and $\alpha+\beta>0$) 
can be expressed as
\begin{equation}\label{eq:phiG}
 \phi (x)=\int_\Omega G (x,y) \rho(y)\, dy
 \end{equation}
for some Green kernel $G (x,y): \Omega\times\Omega\to\R^d$.
Importantly, the Green kernel is not of the form $G(|x-y|)$. But since the equation \eqref{eq:phi0}
is self-adjoint, we have
$$
G(x,y)=G (y,x), \qquad \na_x G (x,y)=-\na_y G (y,x).
$$
In addition, the maximum principle applied to \eqref{eq:phi0}  gives 
\begin{equation}\label{eq:phimax} 
0\leq  \phi(x)\leq 1 \qquad \mbox{ in } \Omega,
\end{equation}
and multiplying \eqref{eq:phi0} by $\phi$ and integrating leads to the estimate
\begin{equation}\label{eq:phiH1}
 \| \phi \|_{L^2(\Omega)}^2+   \|\na \phi\|_{L^2(\Omega)}^2 \leq 1.
\end{equation}

Most results presented in this paper (with the exception of Theorem \ref{thm:charac})  are not specific to equation \eqref{eq:phi0} and hold for more general kernel $G$. We focus on equation \eqref{eq:phi0} because of its interest in several applications, most notably chemotaxis models. In addition, it would be straighforward to generalize our results to include an external potential $\phi_e\in L^\infty(0,\infty;W^{2,\infty}(\Omega))$ and replace \eqref{eq:phiG} with
\begin{equation}\label{eq:phiG1} 
 \phi (x)=\int_\Omega G (x,y) \rho(y)\, dy + \phi_e (x)
\end{equation}
(Theorem \ref{thm:charac} would then require the additional assumption that $-\Delta \phi_e>0$).

\medskip

We are studying the  following initial boundary value problem with $\mu\geq 0$:
\begin{equation}\label{eq:weak}
\begin{cases}
\pa_t \rho  -\mu \Delta \rho + \div( \rho \na \phi-\na p)=0 ,\qquad \mbox{ in } \Omega\times(0,\infty), \qquad p\in P(\rho)\\
( -\mu \na \rho + \rho \na \phi-\na p)\cdot n = 0, \qquad  \mbox{ on }\pa \Omega\times(0,\infty)\\
\rho(x,0) = \rho_{in}(x)\qquad  \mbox{ in } \Omega
\end{cases}
\end{equation}
with $\phi$ solution of \eqref{eq:phi0}. 
The Neumann boundary condition is natural and guarantees that $\int_\Omega \rho(t)\, dx$ is preserved. 
We recall (see the introduction) that the pressure $p(x,t)$, as in the equations of incompressible fluid mechanics, is a Lagrange multiplier for the incompressibility constraint $\rho\leq 1$. In particular (see \eqref{eq:P})
 the condition $p\in P(\rho)$ implies $\rho\leq 1$, $p\geq 0$ and $p(1-\rho)=0$ a.e.  in  $\Omega\times(0,\infty)$.
\medskip




We will use the following natural definition of weak solutions of \eqref{eq:weak}. 
\begin{definition}\label{def:weak}
The pair of functions $(\rho,p)$ is a weak solution of \eqref{eq:weak}  if 
$\rho \in L^1(0,\infty;L^\infty(\Omega)) \cap C^{1/2}(0,\infty;\P(\Omega))$, $p\in L^2(0,\infty;H^1(\Omega))$ with
$$0\leq \rho\leq 1,\quad  p\geq 0, \quad (1-\rho) p=0 \quad\mbox{ a.e. in } \Omega_T$$
and the followings hold:
\begin{equation}\label{eq:weak11}
\int_\Omega \rho_{in} (x) \zeta(x,0)\, dx + \int_0^\infty \int_\Omega (\rho\, \pa_t\zeta + \rho v \cdot \na \zeta)\, dx dt = 0 
\end{equation}
for any function $\zeta\in C^\infty_c(\overline \Omega\times [0,\infty))$ and for some $v\in (L^2(d\rho))^d$ satisfying 
\begin{equation}\label{eq:weak12}
\int_0^\infty \int_\Omega  (\rho v\cdot \xi - \rho\na \phi \cdot \xi -\mu\rho\, \div\xi - p\, \div \xi) \, dx \, dt= 0
\end{equation}
for any vector field  $\xi \in C^\infty_c(\overline \Omega\times (0,\infty) ; \R^d)$ such that $\xi \cdot n=0$ on $\pa\Omega$ and with $\phi$ given by \eqref{eq:phiG}.
\end{definition}

Here we denote  $L^2(d\rho) :=\{f:  \int_0^\infty \int_\Omega  |f(x,t)|^2 \rho(x,t)\, dx\, dt<\infty\}$. 
\medskip

Equation \eqref{eq:weak11} is the usual weak formulation for the continuity equation $\pa_t \rho + \div (\rho v)=0$ with Neumann boundary conditions and initial condition $\rho_{in}$.  \eqref{eq:weak12} is equivalent to $\rho v =  -\mu\na \rho+\rho \na \phi - \na p$ in $L^{2}(\Omega_T)$.
Recall that $\P(\Omega)$ is equipped with the Wasserstein distance $W_2$, so the condition
$\rho \in C^{1/2}(0,\infty;\P(\Omega))$ means that $W_2(\rho(t),\rho(s)) \leq C|t-s|^{\frac1 2 }$. In view of Lemma \ref{lem:tech}, this implies also that $\rho \in C^{1/2}(0,\infty;H^{-1}(\Omega))$.

\medskip

\subsection{The JKO scheme}
The key tool to prove the existence of weak  solutions in the sense of Definition~\ref{def:weak} is the fact that  \eqref{eq:weak} is the gradient flow - or minimizing movement scheme - with respect to the Wasserstein distance $W_2$ for the functional
$$ J (\rho)= \mu \int \rho \log\rho \, dx - \frac 1 {2}  \int_\Omega \rho\, \phi  \, dx  = \mu \int \rho \log\rho \, dx - \frac 1 {2 }  \int_\Omega \int_\Omega G (x,y)\rho(x)\rho(y)\, dy\, dx $$
with the constraint $\rho\leq 1$ (which will be enforced by requiring that $\rho\in K$, with $K$ defined by \eqref{eq:K}).

\medskip

The idea of the JKO scheme is to construct a time-discrete approximation of the solution by successive applications of a minimization problem. More precisely, for a given initial data $\rho_{in}\in K$,  we fix a time step $\tau>0$ (destined to go to zero) and define the sequence $\rho^n $ by:
\begin{equation}\label{eq:JKO}
\rho^0=\rho_{in} , \qquad  \rho^n \in \mathrm{argmin}\left\{  \frac {1}{2\tau} W_2^2(\rho,\rho^{n-1}) + J(\rho)\, ;\, \rho\in K\right\} \qquad \forall n\geq 1.
\end{equation}

The fact that this problem has a minimizer will be proved in Proposition \ref{prop:minchar} (and the minimizer is unique if $\tau$ is small enough). Furthermore,  if $T^n$ is the unique optimal transport map from $\rho^n$ to $\rho^{n-1}$ (that is $T^n\#\rho^n=\rho^{n-1}$ and $W_2^2(\rho^n,\rho^{n-1}) = \int|x-T^n(x)|^2 \rho^n dx$), we can define the velocity
$$ v^n(x) =\frac{x-T^n(x)}{\tau}$$
and the pressure variable $p^n(x)$ such that
$$ 
\rho^n v^n =-\mu \na \rho^n + \rho^n \na \phi^n - \na p^n, \qquad p^n \in H^1_{\rho^n}
$$
(the existence of $p^n$ will be proved in Proposition \ref{prop:pressure}) with $\phi^n = \int_\Omega G(x,y)\rho^n (y)\, dy.$

\medskip

We can then define the piecewise constant function $\rho^{\tau},p^{\tau}:[0,T]\mapsto P(\Omega)$ by
\begin{equation}\label{eq:interrhop}
\begin{array}{lll}
\rho^{\tau}(t)&:=\rho^{n+1} &\text{ for all }  t\in[n\tau,(n+1)\tau)\\
p^{\tau}(t)& :=p^{n+1} &  \text{ for all }  t\in[n\tau,(n+1)\tau).
\end{array}
\end{equation}
Our goal is now to characterize the limit of these functions when $\tau\to0$ and  prove the existence of weak solutions to \eqref{eq:weak}-\eqref{eq:phi0})

\subsection{Main results}
Our first theorem concerns the limit $\tau\to0$ and proves in particular the existence of a weak solution of  \eqref{eq:weak}:

\begin{theorem}[Convergence when $\tau\to0$]\label{thm:conv}
Fix $T>0$ and $\mu\geq 0$. Let $\tau_k$ be a sequence such that $\tau_k\to 0$.
For any initial condition $\rho_{in} \in K$, there exists a subsequence, still denoted $\tau_k$, such that
the interpolations $\rho^{\tau_k }$ defined by \eqref{eq:interrhop}  converge uniformly in $[0,T]$ with respect to $W_2$ to $\rho $  and $p^{\tau_k }$ converges weakly in $L^2(0,T;H^1(\Omega))$ to $p$ where $(\rho ,p )$ is a weak solution of \eqref{eq:weak} in  the sense of Definition \ref{def:weak}.
Furthermore, $\rho$ satisfies the energy inequality
\begin{equation}\label{eq:energy}
J(\rho(t)) + \int_0^t \int_\Omega |v|^2 d\rho \leq J(\rho_{in}) \quad \forall t>0
\end{equation}
with $v$ defined in Definition \ref{def:weak}.
\end{theorem}
Since the limit is proved only for a subsequence, an immediate question is whether the equation \eqref{eq:weak} is enough to characterize $\rho(t)$ for all $t>0$. 
Such a result was proved in \cite{PQV} for a related equation (with no drift term, but with a growth term) using a delicate duality argument. 
We will prove that a similar result can be obtained in our case:
\begin{proposition}\label{prop:unique}
For $\mu\geq 0$ and 
given $\rho_{in}\in \P(\Omega)$, there exists a unique $(\rho,p)$ weak solution of \eqref{eq:weak} in  the sense of Definition \ref{def:weak}.
In particular, the whole sequence (and not just a subsequence) converges in Theorem \ref{thm:conv}.
\end{proposition}
The proof  of Proposition \ref{prop:unique} would be a relatively straightforward adaptation of \cite{PQV} if the potential $\phi$ was in $W^{2,\infty}(\Omega)$. Unfortunately, Calderon-Zygmund estimates only give $\phi \in L^\infty (0,\infty;W^{2,p}(\Omega) )$ for all $p<\infty$ in our case. 
Particular care will have to be taken to handle the drift term because of this fact (see Section \ref{sec:unique}).
Uniqueness for the Keller-Segel model without the density constraint is proved in \cite{CaLiMa14}. 
In \cite{DiMe16}, uniqueness is proved for an advection-diffusion equation with density constraint, under some monotonicity assumptions on the (given) vector field. Very recently the uniqueness for a similar system with Newtonian kernel in $\R^n$ and $\mu=0$ was shown by a different approach in \cite{HLP}.

\medskip

Before stating our next result, we introduce, for a given density $\rho\in K$ the pressure set
\begin{equation}\label{eq:H1rho}
H^1_\rho := \left\{ q\in H^1(\Omega) \, ;\, q\geq 0, \; q(1-\rho)=0 \mbox{ a.e. } \right\}.
\end{equation}
We recall that weak solutions of \eqref{eq:weak} in  the sense of Definition \ref{def:weak} are in particular H\"older continuous in time with respect to $W_2$. It follows (see Lemma \ref{lem:tech}) that the function
$t\mapsto \int_\Omega \rho(x,t)\zeta(x)\, dx$
is continuous for all $\zeta\in H^1(\Omega)$ and thus well defined for all $t>0$.
It is thus possible to define the set $H^1_{\rho(t_0)}$ for all $t_0$ (even though $\rho(x,t)$ is only defined for a.e. $x,t$)  as the set of $q\in H^1(\Omega)$ satisfying $q\geq 0$ and $\int_\Omega (1-\rho(x,t_0))q(x)\, dx=0$.
With this, we can now state our main result concerning the pressure $p(t_0)$:
\begin{proposition}\label{prop:obs}
With the notations of Theorem \ref{thm:conv} and for any $t_0>0$, let $q(t_0)$ be the unique solution of the variational problem (obstacle problem):
\begin{equation}\label{eq:obst0}
\begin{cases}
q \in H^1_{\rho (t_0)} \\
\displaystyle  \int [\na q-\na \phi(t_0)]  \cdot [\na q - \na \zeta]  \, dx \leq 0\qquad \forall \zeta \in H^1_{\rho (t_0)}
 \end{cases}
\end{equation}
where $\phi (t_0)$ is the solution of \eqref{eq:phi0} with $\rho=\rho (t_0)$. 
Then the limiting pressure satisfies $ p (t) = q(t)$ a.e. $t>0$, and $\na p^\tau$ converge to $\na q$ strongly in $L^2(\Omega_T)$.

\end{proposition}
After redefining $p(t)$ on a set of measure zero, we can thus assume that $p(t)=q(t)$ for all $t>0$.

\medskip

    Note that formally \eqref{eq:obst0} states $-\Delta (q- \phi (t_0))\geq 0$ in $\{\rho(t_0)=1\}$, with equality in $\{q>0\}$.
An important ingredient in the proof of Proposition \ref{prop:obs} will be the fact, interesting in itself, that the discrete pressure $p^\tau(t)$ given in \eqref{eq:interrhop} is a subsolution of  \eqref{eq:obst0}(with $\rho^\tau$ instead of $\rho$), for all $t>0$ and all $\tau>0$. (see Lemma \ref{lem:pressureineq} and Remark \ref{rmk:pressureineq}).
Moreover, the discrete solution of the obstacle problem is not much larger from $p^{\tau}$. Indeed we can show that the discrete solution of the obstacle problem converges, like $p^\tau$, to the solution of \eqref{eq:obst0}:
\begin{proposition}\label{prop:obs2}
For all $\tau>0$ and $t_0>0$, let $q^\tau(t_0)$ be the unique solution of the 
\begin{equation}\label{eq:obst0tau}
\begin{cases}
q \in H^1_{\rho^\tau (t_0)} \\
\displaystyle  \int [\na q-\na \phi^\tau(t_0)]  \cdot [\na q - \na \zeta]  \, dx \leq 0\qquad \forall \zeta \in H^1_{\rho^\tau (t_0)}
 \end{cases}
\end{equation}
Then $p^\tau \leq q^\tau$. Moreover $\na q^\tau$ converges strongly in $L^2(\Omega_T)$ to   $\na q$ (with $q$ solution of \eqref{eq:obst0}). In particular, we have
$$\| \na (p^\tau-q^\tau)\|_{L^2(\Omega_T)} \to 0.$$
\end{proposition}
    
\medskip

If the set $\{\rho(t_0)=1\}$ is smooth enough, then \eqref{eq:obst0} is the usual variational formulation of an obstacle problem in the set $\{\rho(t_0)=1\}$ with zero Dirichlet boundary condition on $\pa \{\rho(t_0)=1\}\cap \Omega$ and Neumann condition on $\pa\Omega \cap \{p>0\}$.
Using the fact that $\Delta \phi (t_0)< 0$ in $\{\rho(t_0)=1\}$ we can in fact show (still assuming that $\{\rho(t_0)=1\}$ is smooth enough) that  the solution of \eqref{eq:obst0} solves
$$
\begin{cases}
\Delta  q= \Delta  \phi  & \mbox{ in } \{\rho (t_0)=1\}   \\
q=0 &  \mbox{ on } \pa \{\rho (t_0)=1\}  \cap \Omega,\\
(\na \phi-\na q)\cdot n = 0  & \mbox{ on }\pa \Omega\cap  \{\rho (t_0)=1\} 
\end{cases}
$$
which makes rigorous the equation for $p $ in \eqref{eq:introHS}.
However, it is difficult to show that  the set $\{\rho (t_0)=1\}$ is regular (in fact, we do not know how to prove that its boundary has zero Lebesgue measure) which is why the variational formulation \eqref{eq:obst0} is useful.
\medskip

All the results presented above  hold if we add an external potential $\phi_e\in L^\infty(0,\infty;W^{2,\infty}(\Omega))$ by replacing \eqref{eq:phiG} with \eqref{eq:phiG1}. In that case, and depending on the sign of $\Delta \phi_e$, it is possible for the unique solution of \eqref{eq:obst0} to have support $\{q>0\} \neq \{\rho(t_0)=1\}$. This corresponds to an instant collapse of the saturated set $ \{\rho(t)=1\}$ similar to the phenomenon observed and studied in \cite{GKM}.

\medskip

The variational formulation \eqref{eq:obst0} implies in particular the following:
\begin{corollary}\label{cor:pressure}
The pressure $p (x,t)$ satisfies the complementarity condition 
\begin{equation}\label{eq:compcond}
 p ( \Delta p  -  \Delta \phi ) = 0 \mbox{ in }\mathcal D'(\Omega_T)
 \end{equation}
and the velocity $v =  \rho  \na \phi - \na p $ satisfies  $v(t) = P_{C(\rho (t))}( \na \phi (t))$ where $P_{C(\rho)}$ denotes the orthogonal projection onto the set of admissible velocity defined in Appendix \ref{sec:projection}.
\end{corollary}

\begin{remark} 
The key estimates used to prove Theorem \ref{thm:conv} (see Lemma  \ref{lem:unifestimates1} and \ref{lem:tech}) are independent of $\mu$.
Furthermore, Proposition \ref{prop:obs} implies that $\nabla p$ is  uniformly bounded in  $L^2(\Omega_T)$  with respect to  $\mu$.  
We can thus proceed as in the proofs of the results above to show that the weak solution of \eqref{eq:weak} with $\mu>0$  (constructed in Theorem \ref{thm:conv}) converges, as $\mu\to0$ to the weak solution of \eqref{eq:weak} with $\mu=0$ (with the gradient of the pressure converging in $L^2$ strong). In particular, the 
Richards equation 
\eqref{eq:Rich} is an approximation of the Hele-Shaw problem \eqref{eq:HS1}-\eqref{eq:HS2} when $\mu\ll1$.
\end{remark}

\medskip

Finally, an important question, when working with the weak equation \eqref{eq:weak} and the related Hele-Shaw problem \eqref{eq:introHS}   is whether the solution of \eqref{eq:weak} is a characteristic function $\chi_{\Omega_s(t)}$ for all time, or takes value in $(0,1)$.
Indeed, it is only when $\rho=\chi_{\Omega_s(t)}$ that we can claim that  \eqref{eq:weak}  is related to the Hele-Shaw problem \eqref{eq:introHS}. Otherwise, the evolution of $\rho$ in the region $\{0<\rho <1\}$ must be taken into account.

Such a property clearly does not hold when $\mu>0$. When $\mu=0$, it can be proved that the discrete  approximation $\rho^{\tau}$ is always a characteristic function. This   result  was first proved in \cite{JKM} (we recall the proof in appendix).
However, this property is typically lost when $\tau\to0$ without additional conditions. 
Assuming that the initial condition $\rho_{in}=\chi_{E_{in}}$ is a characteristic function (and that $\mu=0$) we can (almost) prove that $\rho$  is characteristic function for all time.
A similar result was proved in \cite{CKY} when $\phi$ is a fixed potential satisfying $-\Delta \phi >0$ by using viscosity solution type arguments.
In our case, we only have $-\Delta \phi >0$ in the set $\{\rho=1\}$ but this is enough to recover the result of \cite{CKY}  and we will prove (using a completely different approach from \cite{CKY}):
\begin{theorem}\label{thm:charac}
Assume that the initial condition $\rho_{in}=\chi_{E_{in}}$ is a characteristic function and that $\mu=0$.
Let $(\rho,p)$ be the solution of \eqref{eq:weak} provided by Theorem \ref{thm:conv}.

For all $t>0$ there exists a set $\Omega_s(t)$ such that
\item(i) $\rho(x,t)=1$ a.e. in $\Omega_s(t)$
\item(ii) $\rho(x,t) = 0 $ a.e. in $\Omega \setminus \overline{\Omega_s(t)}$.
\end{theorem}
We point out that, as in \cite{CKY}, we do not exactly prove that $\rho= \chi_{\Omega_s(t)}$ a.e., since we do not know that $\pa \Omega_s(t)$ has zero Lebesgue measure. This is a very delicate (and important) regularity question that we do not address in this paper.
Nevertheless, together with Proposition \ref{prop:obs}, Theorem \ref{thm:charac} makes more precise the connection between the incompressible chemotaxis model without diffusion (\eqref{eq:weakHS} with $\mu=0$) and the Hele-Shaw free boundary problem \eqref{eq:introHS}.
This result also holds when $\phi$ is given by \eqref{eq:phiG1}, provided the external potential $\phi_e$ satisfies $-\Delta \phi_e\geq 0$.



\medskip

\section{The minimization problem }
\subsection{Existence of a minimizer}
In this section, we will prove that the JKO scheme defined by \eqref{eq:JKO}
 is well defined and we study the properties of the minimizers.
We recall that the energy functional is defined by
$$ J(\rho) =\mu \int \rho \log \rho\, dx -   \frac 1 {2}  \int_\Omega \rho \phi \, dx  $$
with $\mu\geq 0$
and for a given $\brho \in \P(\Omega)$, we consider the minimization problem
\begin{equation}\label{eq:min1}
\min\left\{  \frac {1}{2\tau} W_2^2(\rho,\brho) + J(\rho)\, ;\, \rho\in K\right\}
\end{equation}
where $K$ is the convex set defined by \eqref{eq:K}.

\medskip

We then have:
\begin{proposition}\label{prop:minchar}
The following holds for all $\brho\in \P(\Omega)$:
\item(i) The set of minimizers of \eqref{eq:min1} is non empty 
\item (ii) When $\mu=0$, any minimizer of \eqref{eq:min1}  is a characteristic function.
\item (iii) When $\mu>0$ then any minimizer $\rho^*$ of  \eqref{eq:min1} satisfies $\rho^*(x)>0$ a.e. in $\Omega$ and $\log(\rho^*) \in L^1(\Omega)$.
\end{proposition}

The existence of a minimizer (part (i)) follows from Lemma \ref{lem:cont} below.
The fact that any minimizer is a characteristic function when $\mu=0$ (part (ii)) is a very nice result which was first proved in \cite{JKM} in a slightly different setting;
we recall the proof in appendix for the sake of completeness.
Finally, we refer to \cite{OTAM} Lemma 8.6 for a proof of (iii). 

\begin{lemma}\label{lem:cont}
The functional $J$ is bounded below and lower semicontinuous   with respect to the weak convergence in $\mathcal P(\Omega)$.
\end{lemma}

\begin{proof}[Proof of Lemma \ref{lem:cont}]
Since $\phi \leq 1$ (see \eqref{eq:phimax}), $\rho\leq 1$ and $s\log s\geq -e^{-1}$ for $s\geq 0$, we have $J(\rho) \geq -C$ for all $\rho\in K$. Furthermore, for any sequence  $\rho_k$ which converges weakly to $\rho$ it is easy to show that the corresponding $\phi_k$, which are bounded in $W^{2,p}(\Omega)$ for all $p<\infty$, converge strongly to $\phi$ in $W^{1,p}$ and thus in $C^\alpha$. The lower semicontinuity of $J$ follows from the convexity of $s\mapsto s\log s$.
\end{proof}

\begin{remark}
While not necessary for our analysis, it is worth pointing out that the minimizer is unique for $\tau$ small enough.
This follows from the fact that the functional  $J$ is $\omega$-convex (although it is not $\lambda$-convex).
We refer to \cite{CKY} where the $\omega$-convexity is proved for the corresponding functional when $G(x,y)$
is the usual Newtonian potential.
The proofs are similar and we do not provide details here, since the result isn't necessary for our analysis.
Without using the uniqueness result, one can still define the JKO scheme \eqref{eq:JKO} by choosing for $\rho^n$  any element in the set of minimizers.
\end{remark}

\subsection{The Euler-Lagrange equation}
We recall that given two probability measures $\mu$ and $\nu$, we have the equivalent (dual) definitions of the Wasserstein distance:
\begin{align}
\frac 1 2 W_2^2(\mu,\nu) 
& =\min \left\{\frac 1 2\int_{\Omega\times\Omega} |x-y|^2 d\pi(x,y)\, ;\,  \pi \in \Pi(\mu,\nu)  \right\}\label{eq:W2K1} \\
& = \max \left\{
\int_\Omega \psi(x) \, d\mu + \int_\Omega \vphi(y) \, d\nu\, \Big|\, \psi,\vphi \in C^0(\overline \Omega), \; \psi(x)+\vphi(y) \leq\frac 1 2 |x-y|^2\right\}.
\label{eq:W2K}
\end{align}
The minimum in \eqref{eq:W2K1} is attained for some $\pi$ of the form $\pi = (id,T)_\#\mu$ with $T\# \mu=\nu$ and
$$ \frac 1 2 W_2^2(\mu,\nu)  = \frac 1 2 \int_\Omega |x-T(x)|^2 d\mu.$$
The maximum in \eqref{eq:W2K} is realized by a  pair of Lipschitz conjugate functions $(\psi,\psi^c)$ where
$$
\psi^c(y):= \inf_{x\in\Omega} \frac 1 2 |x-y|^2 - \psi(x).
$$  
Below, we will call Kantorovich potential from $\mu$ to $\nu$ any concave function $\psi$ such that
$(\psi,\psi^c)$ realizes the maximum in \eqref{eq:W2K}.
Note that we do not, in general, have uniqueness of the potential (unless one of the support of the two measures is the whole domain $\overline\Omega$), but we have 
$
T(x) = x-\na \psi(x) $ a.e. $x\in \Omega$, which  is uniquely determined on $\{\mu>0\}$ (and $T^{-1}(x) = x - \na \psi^c(x)$).

In the remainder of this section, we denote by  $\rho^*$ a minimizer of \eqref{eq:min1} provided by Proposition \ref{prop:minchar}. We start with:
\begin{proposition}\label{prop:EL}
For any $\mu\geq 0$,
let $\rho^*$ be a minimizer of \eqref{eq:min1} and $\phi^* (x)= \int_\Omega G(x,y)\rho^*(y)\, dy$ be the corresponding solution of \eqref{eq:phi0}.
There exists a  Kantorovich potential $\psi$ from $\rho^*$ to $\brho$ such that
\begin{equation}\label{eq:pressure1}
\int_{\Omega}\left(\mu \log \rho^* -\phi^*+\frac{\psi}{\tau}\right)(\rho-\rho^*)\geq 0
\end{equation}
for all $\rho \in K$.
\end{proposition}
The proof is similar to the proof of Lemma 3.1 in \cite{MRS}. We provide it here for the reader's convenience.
\begin{proof}
Given $\rho\in K$, we consider $\rho^{\delta}=\rho^*+\delta(\rho-\rho^*)$, which belongs to $K$ for all $\delta\in (0,1)$.
By optimality of $\rho^*$, we have:
$$J (\rho^{\delta})+\frac{1}{2\tau}W_2^2(\rho^{\delta},\brho)\geq J (\rho^*)+\frac{1}{2\tau }W_2^2(\rho^*,\brho).
$$
and so
\begin{equation}\label{eq:opt0}
\liminf_{\delta\to 0} \frac 1 \delta [ J (\rho^{\delta})-  J (\rho^*)] +\frac{1}{2\tau}\liminf_{\delta\to 0}   \frac 1 \delta[ W_2^2(\rho^{\delta},\brho)- W_2^2(\rho^*,\brho)]\geq 0.
\end{equation}
Let $\psi_\delta$ be a Kantorovich potentials from $\rho^\delta$ to  $\brho$. We have:
$$
\frac 1 2 W_2^2(\rho^{\delta},\brho) = \int_\Omega \psi_\delta(x) \rho^\delta(x) \, dx + \int_\Omega \psi_\delta^c(y) \brho(y)\, dy$$
$$
\frac 1 2 W_2^2(\rho^*,\brho) \geq \int_\Omega \psi_\delta(x) \rho^*(x) \, dx + \int_\Omega \psi_\delta^c(y) \brho(y)\, dy$$
and so
$$
\frac{1}{2 \delta }\left(W_2^2(\rho^{\delta},\brho)-W_2^2(\rho^*,\brho)\right)
\leq 
 \frac{1}{\delta} \int_\Omega \psi_\delta(x)( \rho^\delta(x)-\rho^*(x) )\, dx
= \int_\Omega \psi_\delta(x)( \rho(x)-\rho^*(x) )\, dx.
$$
We note that we can always assume that $\psi_\delta(x_0)=0$ for some $x_0\in\Omega$, in which case $\psi_\delta$ converges uniformly, up to a subsequence,  to a Kantorovich potential $\psi$ associated to $\brho$ and $\rho^*$ (see \cite{ButSant}).
We deduce
\begin{equation}\label{eq:opt1}
\liminf_{\delta\to 0 } \frac{1}{2 \delta }\left(W_2^2(\rho^{\delta},\brho)-W_2^2(\rho^*,\brho)\right)
\leq  \int_\Omega \psi (x)( \rho(x)-\rho^*(x) )\, dx.
\end{equation}
Furthermore, the fact that 
$\mu \log(\rho^*) \in L^1(\Omega)$ (see Proposition \ref{prop:minchar} (iii)) implies that $\delta \to J (\rho^{\delta})$ is differentiable at $0$ and that
$$  
\lim_{\delta\to 0 }\frac 1 \delta [J (\rho^{\delta})-J (\rho^*)] = \delta J(\rho^*) [\rho-\rho^*]= - \int_{\Omega}{\phi^*(\rho-\rho^*)}\, dx + \mu \int_\Omega \log \rho^* (\rho-\rho^*)\,. dx
$$
Together with \eqref{eq:opt0} and \eqref{eq:opt1} this implies \eqref{eq:pressure1}.

\end{proof}

\subsection{The velocity}
Let $T$ be the unique optimal transport map from  $\rho^*$ to $\brho$. We define the velocity $v(x)$ by
$$ v(x)  = \frac{x-T(x)}{\tau}  = \frac{\na \psi (x) }{\tau}.$$
We then have the following classical result:
\begin{proposition}\label{prop:W}
The velocity $v$ satisfies:
\begin{equation}\label{eq:vbd}
 \int_\Omega \rho^* v^2\, dx = \frac{1}{\tau^2} W_2^2(\rho^*,\brho) <\infty
 \end{equation}
and  for all test function $\zeta \in C^{2}(\Omega)$:
\begin{equation}\label{eq:vzeta}
\tau \int_\Omega \rho^*(x)v(x) \cdot \na \zeta(x) \, dx =  \int_\Omega [\rho^*(x)-\brho(x)]\, \zeta (x) \, dx + \mathcal O(\| D^2\zeta\|_\infty W_2^2(\rho^*,\brho)).
 \end{equation}
\end{proposition}
\begin{proof}
The definition of $T$ gives
$$ T\# \rho^* = \brho \mbox{ and } W_2^2(\rho^*,\brho) = \int |x-T(x)|^2 \rho^*(x)\, dx=\tau^2 \int |v(x)|^2 \rho^*(x)\, dx.$$
We get \eqref{eq:vbd} immediately and we can write:
\begin{align*}
\int(\rho^*(x)-\brho(x)) \zeta(x) \, dx 
& = \int [\zeta(x) -\zeta(T(x)) ] \rho^*(x) \, dx\\
& = \int \na \zeta(x) \cdot(x-T(x))  \rho^*(x) \, dx +\mathcal O\left( \| D^2\zeta\|_\infty  \int_\Omega |x-T(x)|^2 \rho^*(x)\, dx\right) \\
& = \tau\int\na\zeta (x) \cdot v(x)  \rho^*(x) \, dx +\mathcal O\left(  \| D^2\zeta\|_\infty  W_2^2(\rho^*,\brho)  \right)
 \end{align*}
and the result follows.
\end{proof}

\subsection{The pressure}
To end this section, we prove that the Euler-Lagrange equation \eqref{eq:pressure1} can be rewritten using a pressure term:
\begin{proposition}\label{prop:pressure}
Let $\rho^*$ be a minimizer of \eqref{eq:min1} and $\phi^* (x)= \int_\Omega G(x,y)\rho^*(y)\, dy$ be the corresponding solution of \eqref{eq:phi0}.
There exists $p\in H^1_{\rho^*}$ such that
\begin{equation}\label{eq:gradp}
- \na p(x) =
\begin{cases} 
  - \na \phi^* (x)+ v(x)  & \mbox{ a.e. in } \{\rho^*=1\}\\
 0 & \mbox{ a.e in } \{ \rho^*<1\}.
 \end{cases}
\end{equation}
and 
\begin{equation}\label{eq:EL5}
\rho^* v 
=  -\mu \na \rho^*+ \rho^*\na \phi^* -  \na p \quad \mbox{ a.e. in } \Omega
\end{equation}
In addition  $\mu \log \rho^*\in H^1$ and 
if $\mu=0$ then $-\na p= - \na \phi^* + v$ a.e  in $\{\rho^*>0\}$.
\end{proposition}
\begin{proof}
We can proceed as in \cite{MRS}: 
If we introduce $F(x) = \mu \log \rho^* -\phi^*(x)+\frac{{\psi}(x)}{\tau} $, Proposition \ref{prop:EL} gives
\begin{equation}\label{eq:Fmin} \int_\Omega  F\rho^* \, dx \leq \int_{\Omega} F\rho\, dx \qquad \forall \rho\in K.
\end{equation}
This implies that there exists a Lagrange multiplier $\ell$ (associated to the mass constraint) such that
\begin{equation}\label{eq:L}
\begin{cases}
\rho^*=0 & \mbox{ a.e. in } \{F>\ell\} \\
0\leq \rho^*\leq 1 & \mbox{ a.e. in } \{F=\ell\} \\
\rho^*=1 & \mbox{ a.e. in } \{F<\ell\} 
\end{cases}
\end{equation}
We now define the pressure by
$$ p(x) = (\ell-F(x))_+.$$
which satisfies $(1-\rho^*(x))p(x)=0$ a.e. in $\Omega$. 

When $\mu=0$, we then get that $p\in H^1_{\rho^*}$ thanks to \eqref{eq:vbd} and \eqref{eq:phiH1}.

When $\mu>0$, we first note that 
\begin{equation}\label{eq:Fell}
\{F<\ell\} = \{ \ell+\phi^* - \psi/\tau >0\}.
\end{equation}
Indeed, we recall (Proposition \ref{prop:minchar} (iii)) that $\rho^*>0$ a.e. in $\Omega$ and so $F \leq \ell$ a.e. (this is also evident here since we have $F=-\infty$ a.e. in $\{\rho^*=0\}$ and so \eqref{eq:L} yields $|\{F> \ell\} |=0$). Furthermore, in $\{F<\ell\}$ we have $\rho^*=1$ and thus $F = -\phi^*(x)+\frac{{\psi}(x)}{\tau} <\ell$ while  in $\{F=\ell\}$ we have $-\phi^*(x)+\frac{{\psi}(x)}{\tau} = \ell  -\mu \log\rho^* >\ell   $.
Using \eqref{eq:Fell}, it is now straightforward to show that 
$$ \mu\log\rho^* = -(-\ell-\phi^*+\psi/\tau)_+ \in H^1\quad \mbox{ and }\quad
  p = \left(\ell  +\phi^*-\psi/\tau  \right)_+ \in H^1.$$


In particular $F\in H^1$ and so $\na F \chi_{\{F<\ell\}} = \na F \chi_{\{F\leq\ell\}}$. We deduce (using \eqref{eq:L}):
$$ 
\na p = - \na F \chi_{\{F<\ell\}} = - \rho^*\na F \chi_{\{F<\ell\}}= - \rho^*\na F \chi_{\{F\leq\ell\}}= - \rho^*\na F
$$ 
Equalities \eqref{eq:gradp} and \eqref{eq:EL5} follow from the definition of $F$.



\end{proof}

Finally, we prove the following lemma, which will be crucial to the derivation of the obstacle problem for the pressure $p$:
\begin{lemma}\label{lem:pressureineq}
Let $\rho^*$ be a minimizer of \eqref{eq:min1},  $\phi^* (x)$ the corresponding solution of \eqref{eq:phi0} and $p$ the pressure given by Proposition \ref{prop:pressure}. There holds:
\begin{equation}\label{eq:pressureobst2}
\int_\Omega \na \zeta(x)\cdot [ \na p(x)- \na \phi^*(x)]   dx \leq 0 \qquad \forall \zeta \in H^1_{\rho^*}.
\end{equation}
\end{lemma}

\begin{proof}

We recall that $T$ is the optimal transportation map such that $T\#\rho^* =\bar\rho$ and we denote
$$ T_t(x) = (1-t)x+t T(x), \qquad \rho_t(x) = T_t \#\rho^*.$$
We point out that $\rho_0=\rho^*$ and $\rho_1=\bar \rho$, so compared to the interpolation $\tilde \rho^\tau(t)$ defined by \eqref{eq:interp}, this $\rho_t$ flows backward.
We have in particular $\rho_t \in \mathcal P(\Omega)$ and since $M\mapsto |\det M|^{-1}$ is convex on the set of positive definite matrices, we have
$$  \rho_t(x) =   \frac{\rho^*}{|\det \na T_t|} \circ T_t^{-1} \leq \left( (1-t) \rho^* + \frac{t\rho^*}{|\det \na T|} \right) \circ T_t^{-1} .$$
Using the fact that $ \bar \rho (T(x)) =  \frac{\rho^*(x)}{|\det \na T(x)|} \leq 1$ a.e. $x\in\Omega$, we deduce 
 $  \rho_t(x) \leq 1 $ a.e. and so $\rho_t\in K$.

Next, we note that for any text function $\zeta\in H^1_{\rho^*}$, we have 
$$ \int_\Omega \zeta(x) ( \rho(x)-\rho^*(x)) \, dx =  \int_\Omega \zeta(x)( \rho(x)-1) \, dx \leq 0 \qquad \forall \rho\in K.$$
In particular, we can write
$$ \int_\Omega \zeta(x) ( \rho_t(x)-\rho^*(x)) \, dx= \int_\Omega [\zeta(T_t(x))-\zeta(x)] \rho^*(x) \, dx \leq 0 .
$$
Dividing by $t>0$ and passing to the limit $t\to0$ we get
\begin{equation}\label{eq:proj-v}
\int_\Omega \na \zeta(x) \cdot (T(x)-x) \rho^*(x)  dx  
= -\tau \int_\Omega \na \zeta(x) \cdot \rho^*(x)  v (x)  dx \leq 0.
\end{equation}
Since $\na \zeta$ satisfies $\na \zeta (1-\rho^*) =0$ a.e., we can now use \eqref{eq:gradp} to get \eqref{eq:pressureobst2}.
\end{proof}

\begin{remark}
With the notations of Appendix \ref{sec:projection}, inequality \eqref{eq:proj-v} says that $-v \in C(\rho^*)$ (i.e. $-v$ is admissible for $\rho^*$).
A similar argument, with $\zeta \in H^1_{\bar \rho}$, can be used to prove that $v\circ T^{-1} \in C(\bar \rho)$.
\end{remark}

\begin{remark}\label{rmk:pressureineq}
Since $p\in H^1_{\rho^*}$, we can take $\zeta=p$ in \eqref{eq:pressureobst2} to find:
\begin{equation}\label{eq:pressureobst}
\int_\Omega \na p(x)\cdot [ \na p(x)- \na \phi^*(x)]   dx  \leq 0.
\end{equation}
This inequality will be all that is really needed to prove that $p^\tau$ converges strongly to the solution of the variational inequality \eqref{eq:obst0}. Of course, Lemma \ref{lem:pressureineq} is stronger than that. It implies in particular that $p(\Delta p-\Delta\phi^*)\geq0$ and that $p$  is the subsolution for the obstacle problem
\begin{equation}\label{eq:obtsq}
\begin{cases}
q \in H^1_{\rho^*} \\[5pt]
\displaystyle \int_\Omega [ \na q(x)- \na \phi^*(x)] \cdot  [\na q(x)-\na \zeta(x)]  dx \leq 0 \qquad  \forall \zeta \in H^1_{\rho^*}.
\end{cases}
\end{equation}
If we denote by $q^*$  the unique solution of \eqref{eq:obtsq},
then \eqref{eq:pressureobst2} implies
$$ 0\leq p(x)\leq q^*(x) \mbox{ in } \Omega.$$
(to see this, take $\zeta = (p-q^*)_+$ in \eqref{eq:pressureobst2} and $\zeta = q^*+(p-q^*)_+$ in \eqref{eq:obtsq}).

\end{remark}

\medskip

\section{Definition of $\rho^{\tau},\widetilde \rho^{\tau}$ and a priori estimates}
In this section we derive  classical a priori estimates which will be used to pass to the limit $\tau\to0$.
We recall that we define the sequence $\rho^n$ by successive applications of the minimization problem \eqref{eq:JKO} 
starting with $\rho^0=\rho_{in}$.
We also define the velocity
$$ v^n (x)  = \frac{x-T^n (x)}{\tau}  = \frac{\na \psi ^n (x) }{\tau}$$
and the pressure $p^n$, given by Proposition \ref{prop:pressure}, which satisfies:
\begin{equation}\label{eq:pressuren} 
\rho^n v^n  = -\mu \na \rho^n+ \rho^n \na \phi^n  - \na p^n , \qquad p^n  \in H^1_{\rho^n},\quad \phi^n (x)= \int_\Omega \rho ^n(y) G(x,y)\, dy
\end{equation}
(of course, $\rho^n$, $p^n$, $T^n$, ...  depend on   $\tau$, even though we do not indicate this dependence).

Finally, we define the piecewise constant interpolations 
$\rho^{\tau}(x,t)$, $p^{\tau}(x,t)$, $v^{\tau}(x,t)$ and $\phi^{\tau}(x,t)$ by
\begin{equation}\label{eq:inter}
\begin{aligned}
\rho^{\tau}(t)& :=\rho^{n+1} &\text{ for all } \ \ t\in[n\tau,(n+1)\tau)\\
p^{\tau}(t)& :=p^{n+1}  & \text{ for all } \ \ t\in[n\tau,(n+1)\tau)\\
v^{\tau}(t)& :=v^{n+1} & \text{ for all } \ \ t\in[n\tau,(n+1)\tau)\\
\phi^{\tau}(t) & :=\phi^{n+1}&  \text{ for all } \ \ t\in[n\tau,(n+1)\tau).
\end{aligned}\end{equation}
We also  define the momentum 
$$E ^{\tau}(x,t)=\rho^{\tau}(x,t) v^{\tau}(x,t).$$
Using the results from the previous section, we can easily prove:
\begin{proposition}
For any smooth test function $\zeta(x,t) $ compactly supported in $\Omega\times [0,T)$
and given $N$ such that $N\tau \geq T$, there holds:
\begin{align}
\int_0^\infty \int_\Omega E^{ \tau}  \cdot \na \zeta \, dx \, dt
&  = -\int_\Omega \rho_{in} (x) \zeta(x,0) \, dx
 -\int_0^\infty  \int_\Omega  \rho^{\tau}(x,t) \pa_t \zeta(x,t) \, dx\,  dt \nonumber \\
 & \qquad
 + \mathcal O \left(\|D^2\zeta\|_{L^\infty(\Omega \times\R_+)} \sum_{n=0}^N W_2^2(\rho^n,\rho^{n-1}) 
+\tau \| \pa_t\zeta\|_\infty + \tau T \| \pa^2_t \zeta\|_\infty
 \right)\label{eq:weakepstau}
 \end{align}

For any smooth vector field $\xi(x,t)$ satisfying $\xi\cdot n=0$ on $\pa \Omega$, there holds:
\begin{equation}\label{eq:pressureepstau}\int_0^\infty \int_\Omega E^{ \tau} \cdot \xi\, dx\, dt   = \int _0^\infty \int_\Omega (\rho^{ \tau} \na \phi^{ \tau} \cdot \xi  + \mu \rho^\tau \div \xi + p^{ \tau}\div \xi) \, dx dt.
\end{equation}
\end{proposition}

\begin{proof}

We recall that \eqref{eq:vzeta} implies
$$
\tau \int \rho^n  (x)v ^n (x) \cdot \na \zeta (x,t) \, dx =  \int [\rho ^n(x)-\rho ^{n-1}(x)]\, \zeta (x,t) \, dx + \mathcal O(\| D^2\zeta(\cdot,t) \|_\infty W_2^2(\rho^n,\rho^{n-1})) \qquad \forall n\geq 1
$$
which gives in particular for $n\geq 1$ and $t\in [(n-1)\tau,n \tau)$:
\begin{equation}\label{eq:weakn}
 \tau \int E^{\tau}(x,t) \cdot \na \zeta (x,t) \, dx = \int [\rho^{\tau}(x,t) -\rho^{\tau}(x,t-\tau) ]\, \zeta (x,t) \, dx + \mathcal O(\| D^2\zeta \|_\infty W_2^2(\rho^n,\rho^{n-1})).
 \end{equation}
Integrating with respect to $t\in [(n-1)\tau,n \tau)$ and adding these equalities for $n=1, \dots N$, we easily obtain
\begin{align}
\int_0^\infty \int_\Omega E^{\tau} \cdot \na \zeta \, dx \, dt
&  = -\int_\Omega \rho_{in} \frac 1 \tau \int_0^\tau \zeta(t)\, dt \, dx
 -\int_0^\infty  \int_\Omega  \rho^{\tau}(x,t) \frac{\zeta(x,t+\tau) - \zeta(x,t)}\tau\, dx\,  dt \nonumber \\
 & \qquad
 + \mathcal O \left(\|D^2\zeta\|_{L^\infty(\Omega \times (0,\infty))} \sum_{k=0}^N W_2^2(\rho^n,\rho^{n-1}) 
 \right)\label{eq:weakepstau1}
 \end{align}
 and  \eqref{eq:weakepstau} follows.
 Equation \eqref{eq:pressureepstau} follows   from  \eqref{eq:pressuren}.
\end{proof}

We recognize in \eqref{eq:weakepstau} and \eqref{eq:pressureepstau} approximations of the weak equations \eqref{eq:weak11}-\eqref{eq:weak12}. 
Our next task is thus to derive a priori estimates that will allow us to pass to the limit in  \eqref{eq:weakepstau} and \eqref{eq:pressureepstau}.
This limit will be easier to manage with the help of the continuous in time interpolation which we introduce now:

\medskip

\paragraph{\bf Continuous interpolation.} 
Interpolating between $\rho^{n}$ and $\rho^{n+1}$ along the natural geodesic for the Wasserstein distance, we
define:
\begin{equation}\label{eq:interp}
\widetilde{\rho}^{\tau}(t)=\left(\frac{t-n\tau}{\tau}(Id-T^{n+1})+T^{n+1}\right){\#}\rho ^{n+1} \qquad \mbox{for} \ \ t\in[n\tau,(n+1)\tau)
\end{equation}
where we recall that  $T^{n+1}$ is the optimal transport from $\rho^{n+1}$ to $\rho^{n}$. We define $\widetilde{v} ^{\tau}(t,\cdot)$ as the unique velocity field such that $\widetilde{v} ^{\tau}(t,\cdot) \in Tan_{\widetilde{\rho} ^{\tau}}\mathcal{P}_2(\R^d)$ and $(\widetilde{\rho}^{\tau},\widetilde{v}^{\tau})$ satisfy the continuity equation, that is:
\[
\widetilde{v} ^{\tau}=v ^{\tau}\circ \left(\frac{t-n\tau}{\tau}(Id-T^{n+1})+T^{n+1}\right)^{-1}.
\]
Finally, we defind the momentum
\[
\widetilde{E} ^{\tau}:=\widetilde{v} ^{\tau}\,\widetilde{\rho}^{\tau}.
\]
In particular we have
\[
\partial_t\widetilde{\rho}^{\tau}+\nabla \cdot \widetilde{E}^{\tau}=0
\]
in the sense of distribution on $\Omega_T$.

\subsection{A priori estimates}
We start with the following lemma:
\begin{lemma}\label{lem:unifestimates1}
There exists a constant $C$ depending only on $J(\rho_{in})$ such that for all   $\tau>0$ we have:
\item[(i)]  $J(\rho ^{N}) \leq C$ and $ \sum_{n=0} ^N W_2^2(\rho^n ,\rho^{n-1} ) \leq 2 C \tau $ for all $N\geq 0$ 
\item[(ii)]  $W_2(\rho^{\tau}(t),\rho^{\tau}(s))\leq C\sqrt{t-s+\tau}$ for any $0\leq s\leq t\leq T$  and $\| v^{\tau}\|_{L^2_{\rho^{\tau}}}\leq C$
\item[(iii)] $W_2(\widetilde{\rho}^{\tau}(t),\widetilde{\rho}^{\tau}(s))\leq C\sqrt{t-s}$ for any $0\leq s\leq t\leq T$ and $\| \widetilde v^{\tau}\|_{L^2_{\widetilde\rho^{\tau}}}\leq C$
\item[(iv)] 
$\| E^{\tau}\|_{L^2(\Omega_T)}\leq C $ and $ \| \widetilde E^{\tau}\|_{L^2(\Omega_T)} \leq C $.
\end{lemma}

\begin{proof}[Proof of Lemma \ref{lem:unifestimates1}]
The definition of  $\rho ^{n+1}$, \eqref{eq:JKO}, implies
\begin{equation}\label{eq:energycompetitor}
J (\rho^{n+1})+\frac{1}{2\tau}W_2^2(\rho^{n },\rho ^{n+1})\leq J (\rho ^{n }),
\end{equation}
In particular,
\begin{equation} \label{eq:disstau}
 J (\rho ^{N+1 })+\frac{1}{2\tau} \sum_{n=1}^N W_2^2(\rho ^{n },\rho ^{n+1 }) \leq J(\rho_{in}) \qquad \mbox{ for all $N\geq 0$.}
\end{equation}
Both (i) and the first inequality in (ii) follow easily.
Furthermore, inequality \eqref{eq:vbd} implies (using \eqref{eq:disstau} with $N\tau \geq T$)
\begin{equation*}
\begin{aligned}
\int_0^T{\int_{\Omega}\rho^{\tau}|v^{\tau}|^2dxdt}&\leq \sum_{n=0}^N \tau  \int_{\Omega} \rho ^{n+1 } |v^{n+1 } (x)|^2 dx  
\leq \sum_{n=0}^N\frac 1 \tau W_2^2(\rho ^{n },\rho ^{n+1})
\leq 2 J (\rho_{in}).
\end{aligned}
\end{equation*}
which gives the second part of (ii). 

To prove (iii), we note that $\widetilde{\rho}^{\tau}$, it is an absolutely continuous curve in the Wasserstein space and it is a constant speed geodesic interpolation with the velocity field $\widetilde{v}^{\tau}=\frac{1}{\tau}W_2(\rho^n,\rho^{n+1})$ when $t\in[n\tau,(n+1)\tau)$. Thus if
$n\tau\leq s\leq t<(n+1)\tau$, we have
\[
W_2(\widetilde{\rho}^{\tau}(t),\widetilde{\rho}^{\tau}(s))\leq W_2(\rho^n,\rho^{n+1})\frac{t-s}{\tau}\leq C\frac{t-s}{\sqrt{\tau}}\leq C\sqrt{t-s}
\]
since $0\leq t-s\leq \tau$. If $n\tau\leq s<(n+1)\tau\leq m\tau\leq t<(m+1)\tau$ for some $n+1\leq m$, using Cauchy-Schwartz, we have
\begin{equation*}
\begin{aligned}
W_2(\widetilde{\rho}^{\tau}(t),\widetilde{\rho}^{\tau}(s))&\leq \left(\sum_{k=n}^{m}{W_2^2(\rho^k,\rho^{k+1})}\right)^{1/2}\left(\left(\frac{(n+1)\tau-s}{\tau}\right)^2+m-(n+1)+\left(\frac{t-m\tau}{\tau}\right)^2\right)^{1/2}\\
&\leq \sqrt{C\tau}\sqrt{\frac{t-s}{\tau}}\leq C\sqrt{t-s}
\end{aligned}
\end{equation*}
since $\left(\frac{(n+1)\tau-s}{\tau}\right)^2\leq \frac{(n+1)\tau-s}{\tau}$ and $\left(\frac{t-m\tau}{\tau}\right)^2<\frac{t-m\tau}{\tau}$.

Moreover, we have (using \eqref{eq:energycompetitor}):
\[
\int_0^T{\int_{\Omega}{\widetilde{\rho}^{\tau}|\widetilde{v}^{\tau}|^2 dx}dt}=\int_0^T{\|\widetilde{v}^{\tau}\|^2_{L^2_{\widetilde{\rho}^{\tau}}}dt}=\int_0^T{|(\widetilde{\rho}^{\tau})'|^2_{W^2}(t)dt}=\sum_n{\frac{1}{\tau}W_2^2(\rho ^{n},\rho ^{n+1})}\leq 2 J(\rho_{in}).
\]
Here we use the notation $|\sigma'|(t)$ for the metric derivative of a curve $\sigma$ and $||\sigma'||_{W^2}(t)$ means that this metric derivative is computed according to the distance $W_2$. 
\medskip

Finally, (iv) follows from the bound $\rho^{\tau}\leq1$:
\begin{equation*}\int_0^T{\int_{\Omega} | E^{\tau}|^2 dxdt} =\int_0^T \int_{\Omega} {\rho^{\tau}}^2  | v^{\tau}|^2 dxd t \leq \int_0^T \int_{\Omega} \rho^{\tau} | v^{\tau}|^2 dxd t 
\end{equation*}
and similarly for $\widetilde{E}^{\tau}$.
\end{proof}

Lemma \ref{lem:unifestimates1}  implies:
\begin{proposition}\label{prop:rholim}
For any sequence $\tau_n\to 0$ there exists a subsequence (still denoted $ \tau_n$) along which $\rho^{\tau_n}$ and $\widetilde{\rho}^{\tau_n}$ converge uniformly with respect to $W_2$  to the same limit
and  $E^{\tau_n}$ and $\widetilde{E}^{\tau_n}$ converge weakly in $L^2(\Omega_T)$ to the same limit.
\end{proposition}

\begin{proof}[Proof of Proposition \ref{prop:rholim}]
The equicontinuity estimate $W_2(\widetilde{\rho}^{\tau}(t),\widetilde{\rho}^{\tau}(s))\leq C\sqrt{t-s}$ (Lemma \ref{lem:unifestimates1} (iii)) implies that 
$\widetilde{\rho}^{\tau_n}$ converges uniformly in $[0,T]$ with respect to the $W_2$ distance (up to a subsequence). The curves $\rho^{\tau}$ and $\widetilde{\rho}^{\tau}$ coincide at every time $t$ of the form $n\tau$. The former is constant on every interval $[n\tau,(n+1)\tau)$, whereas the latter is uniformly H\"older continuous of exponent 1/2, which implies 
\begin{equation}\label{eq:W2tr}
W_2(\rho^{\tau}(t),\widetilde{\rho}^{\tau}(t))\leq C\sqrt{\tau}
\end{equation}
This proves that $\rho^{ \tau_n}$ converges uniformly
to the same limit as $\widetilde{\rho}^{ \tau_n}$.

\medskip

Since $\widetilde{E}^{ \tau_n}$ and $E^{ \tau_n }$ are uniformly bounded in $L^2(\Omega_T)$ (Lemma \ref{lem:unifestimates1} (iv)), they have weak-* limits $\widetilde{E}$ and $E$ respectively. It only remains to prove that $\widetilde{E}=E$.
We note that
\begin{align*}
\widetilde{v}^{\tau}& =v^{\tau}\circ \left(\frac{t-n\tau}{\tau}(Id-T^{n+1})+T^{n+1}\right)^{-1}\\
& =v^{\tau}\circ\left(\frac{(n+1)\tau-t}{\tau}(Id-T^{n+1})+Id\right)^{-1}\\
& =v^{\tau}\circ\left(((n+1)\tau-t)v^{\tau}+Id\right)^{-1}.
\end{align*}
For a test function $f\in \mathrm{Lip}(\Omega \times [0,T] ,\R^d)$, we then have:
\[
\int_0^T{\int_{\Omega}{f\cdot \widetilde{E}^{\tau}}}=\int_0^T{dt\int_{\Omega}{f\cdot \widetilde{v}^{\tau}}\widetilde{\rho}^{\tau}dx}=\int_0^T{dt\int_{\Omega}{f\circ(((n+1)\tau-t)v^{\tau}+Id) \cdot v^{\tau}}\rho^{\tau}dx},
\]
which implies
\begin{equation*}
\begin{aligned}
\left|\int_0^T{\int_{\Omega}{f\cdot \widetilde{E}^{\tau}}}-\int_0^T{\int_{\Omega}{f\cdot E^{\tau}}}\right|&\leq \int_0^T{\int_{\Omega}{|f\circ(((n+1)\tau-t)v^{\tau}+Id)-f|\,|v^{\tau}|\rho^{\tau}dx}dt}\\
&\leq \mathrm{Lip}(f)\tau \int_0^T{\int_{\Omega}{|v^{\tau}|^2\rho^{\tau}dx}dt}\leq C  \mathrm{Lip}(f)\tau
\end{aligned}
\end{equation*}
and so $\widetilde{E}^{ \tau_n}-E^{ \tau_n } \to 0$ in $\mathcal D'(\Omega_T)$.
The result follows.
\end{proof}

We end this section with some a priori estimates for the pressure $p^{\tau}$ and the potential $\phi^{\tau}$ which will be used for the proof of Theorem \ref{thm:conv}.
First we have:
\begin{lemma}\label{lem:unifestimates2}
For all $\tau>0$ we have
\begin{equation}\label{eq:pressurebd}
\| \na p^{\tau}(t) \|_{ L^2( \Omega) }\leq \|\na \phi^\tau(t)\| _{L^2(\Omega)} \leq 1\qquad  \mbox{ for all } t>0 .
\end{equation}
Furthermore, there exists $C$ depending only on $J(\rho_{in})$ such that
\begin{equation}\label{eq:densityH1bd}
\| \mu \na \rho^\tau \| _{L^2(\Omega_T)} \leq C.
\end{equation}
\end{lemma}

\begin{proof}[Proof of Lemma \ref{lem:unifestimates2}]
The first inequality in \eqref{eq:pressurebd} follows from \eqref{eq:pressureobst} while the second one 
is a consequence of \eqref{eq:phiH1}.

To prove \eqref{eq:densityH1bd}, we use 
\eqref{eq:EL5} to write
$\mu \na \rho^\tau  
=  -\rho^\tau v+ \rho^\tau \na \phi^\tau -  \na p^\tau$ and so (since $\rho^\tau\leq 1$):
$$ 
\| \mu\na  \rho^\tau \|^2 _{L^2(\Omega_T)} \leq \int_0^T\int_\Omega  \rho^\tau |v^\tau|^2\, dx\,dt +  \| \na \phi^\tau \| ^2_{L^2(\Omega_T)}
+ \|\na p^\tau \|^2 _{L^2(\Omega_T)}
$$
and the result follows.
\end{proof}

In order to get additional estimates on $\phi^{\tau}$, we will use the following classical Lemma (see for instance \cite{MRS}):
\begin{lemma}\label{lem:tech}
Given $\mu, \nu \in \mathcal P(\Omega)$ and for all $f\in H^1(\Omega)$, we have
$$\left|  \int_\Omega f d(\mu-\nu)\right| \leq \| \na f\|_{L^2(\Omega)} W_2(\mu,\nu).$$
\end{lemma}
We then prove:

\begin{lemma}\label{lem:phibd}
There exist $s>0$ and a constant $C$ depending only on $J(\rho_{in}) $ such that the continuous interpolation $\na \widetilde \phi^{\tau}$ (which solves \eqref{eq:phi0} with $\rho=\widetilde \rho^{\tau}$) satisfies
$$ \|  \na \widetilde \phi^{\tau}\|_{C^s(\Omega_T)} \leq C, $$
$$ \|  \widetilde \phi^{\tau}\|_{C^{1/2}((0,T);H^{-1}(\Omega)})\leq C,$$ 
and
$$ \|  \na \widetilde \phi^{\tau}- \na  \phi^{\tau}\|_{L^2(\Omega_T)} \leq C\sqrt\tau,$$
\end{lemma}
\begin{proof}
Let $0<s<t<T$ and $w = \widetilde \phi^{\tau}(x,s)-\widetilde \phi^{\tau}(x,t)$. Using \eqref{eq:phi0}, we see that $w$ solves
$$ w - \Delta w =  \widetilde \rho^{\tau}(x,s)-\widetilde \rho^{\tau}(x,t)$$
together with the Robin boundary conditions.
Multiplying by  $w$, integrating by parts and using Lemma~\ref{lem:tech} we get:
\begin{align*}
\int_\Omega |w|^2 \, dx +   \int_\Omega |\na w|^2\, dx 
& \leq \int  (\widetilde \rho^{\tau}(x,s)-\widetilde \rho^{\tau}(x,t)) w\, dx \\
& \leq \left(\int_\Omega |\na w|^2\, dx\right)^{1/2} W_2(\widetilde \rho^{\tau}(s)-\widetilde \rho^{\tau}(t))
\end{align*}
Using Lemma \ref{lem:unifestimates1} (iii), we deduce
$$
  \left( \int_\Omega |\na w|^2\, dx\right)^{1/2} \leq C\sqrt{t-s}$$
which proves that $ \na \widetilde \phi^{\tau}$ is bounded in $C^{1/2}(0,T;L^2(\Omega))$.
Furthermore, \eqref{eq:phi0} implies that $   \Delta \widetilde \phi^{\tau}$ is bounded in $L^\infty(0,T;L^p(\Omega))$ for all $p\in[0,\infty]$, so Calderon-Zygmund estimates imply that $  \na \widetilde \phi^{\tau}$
is bounded in $L^\infty(0,T;W^{1,p}(\Omega))$ for all $p<\infty$ and so in $L^\infty(0,T;C^s(\Omega))$ for all $s<1$.
Together, these estimates imply the first bound.
\medskip

Note that we can now write 
$$ \widetilde \phi^{\tau}= \widetilde \rho^{\tau}+ \div (  \na \widetilde \phi^{\tau}).$$
We proved above that $  \na \widetilde \phi^{\tau}$ is bounded in $C^{1/2}(0,T;L^2(\Omega))$, and 
 Lemma \ref{lem:unifestimates1} (iii) together with Lemma~\ref{lem:tech} implies that $ \widetilde \rho^{\tau}$ is bounded in $C^{1/2}(0,T;H^{-1}(\Omega))$. The second bound in Lemma \ref{lem:phibd} follows.
\medskip

Finally the function $z= \widetilde \phi^{\tau}(x,t)- \phi^{\tau}(x,t)$ solves
$$  z -   \Delta z =  \widetilde \rho^{\tau}(x,t)- \rho^{\tau}(x,t)$$
and using Lemma~\ref{lem:tech} and \eqref{eq:W2tr} we get
\begin{align*}
\int_\Omega |z|^2 \, dx +   \int_\Omega |\na z|^2\, dx 
& \leq \left(\int_\Omega |\na z|^2\, dx\right)^{1/2} W_2(\widetilde \rho^{\tau}(t)-\rho^{\tau}(t))\\
& \leq \left(\int_\Omega |\na z|^2\, dx\right)^{1/2} \sqrt{\tau}
\end{align*}
and the last bound of Lemma \ref{lem:phibd}  follows.

\end{proof}

\section{Convergence of the discrete time scheme $\tau \to 0$}

\subsection{Proof of Theorem \ref{thm:conv}}


For any sequence $\tau_n\to0$, we can use  
Proposition \ref{prop:rholim} to find a subsequence (still denoted $\tau_n$) such that
$\rho^{\tau_n}, \widetilde \rho^{\tau_n}$ converge uniformly with respect to $W_2$ to $\rho(x,t)$ and $E^{\tau_n}$, $\widetilde E^{\tau_n}$ converge weakly in $L^2$ to $E(x,t)$.
Furthermore, using Lemma \ref{lem:unifestimates2} we can assume, up to another subsequence, that $p^{\tau_n}\to p$, $\na p^{\tau_n} \to \na p$ and $\mu \na \rho^{\tau_n} \to \mu \na \rho$   weakly in $L^2(\Omega_T)$
and finally, Lemma \ref{lem:phibd} implies that
$\na \widetilde \phi^{\tau_n}$ converges uniformly  and $\na\phi^{\tau_n}$ converges strongly in $L^2$ to 
 $\na \phi$ where $\phi = \int G(x,y)\rho(y)\, dy$.

We then have:
\begin{proposition}\label{prop:conv}
The limits $\rho \in C^{1/2}(0,T;\mathcal P(\Omega))$ and $E\in L^2(\Omega_T)$ are such that
\begin{equation}\label{eq:weakE}
\int_0^\infty \int_\Omega \rho(x,t) \pa_t\zeta (x,t) +  E(x,t)  \cdot \na \zeta (x,t) \, dx \, dt
  + \int_\Omega \rho_{in}(x) \zeta(x,0)\, dx = 0
\end{equation}
for all $\zeta\in C^\infty(\overline\Omega \times [0,T] )$ with
\begin{equation}\label{eq:weakEE}
E = \rho \na \phi -\mu\na \rho - \na p, \qquad \phi=\int_\Omega G(x,y) \rho(y)\, dy.
\end{equation}
Furthermore, $E$ is absolutely continuous with respect to $\rho$, and can be written as $E=\rho v$ for some $v\in L^2 (d\rho)$.
\end{proposition}
\begin{proof} 
Equation \eqref{eq:weakepstau}, together with the a priori estimate Lemma \ref{lem:unifestimates1} (i) yields, for any $\zeta\in C^\infty(\overline{\Omega_T})$:
\begin{align*} 
\int_0^\infty \int_\Omega E^{\tau_n}  \cdot \na \zeta \, dx \, dt
&  = -\int_\Omega \rho_{in} (x) \zeta(x,0) \, dx
 -\int_0^\infty  \int_\Omega  \rho^{\tau_n}(x,t) \pa_t \zeta(x,t) \, dx\,  dt \nonumber \\
 & \qquad
 + \mathcal O \left( (\|D^2\zeta\|_{L^\infty(\Omega \times (0,\infty))} 
+\| \pa_t\zeta\|_\infty +  T \| \pa^2_t \zeta\|_\infty)\tau_n  \right)
 \end{align*}
and passing to the limit, we deduce \eqref{eq:weakE}.

In order to pass to the limit in
\eqref{eq:pressureepstau} and derive \eqref{eq:weakEE}, we just need to check that we can pass to the limit in the nonlinear term $\rho^{\tau_n} \na \phi^{\tau_n} $.
We recall that $\na \phi^{\tau_n} $ converges strongly in $L^2$. And since $\rho^{\tau_n}$ is bounded in $L^1\cap L^\infty (\Omega_T)$ (thanks to the constraint $\rho\leq1$), it converges weakly in $L^2$.
We can thus pass  to the limit in \eqref{eq:pressureepstau}  to get
$$
 \int_0^\infty \int_\Omega E(x,t) \cdot \xi(x,t)\, dx\, dt   = \int _0^\infty \int_\Omega\rho(x,t)  \na \phi (x,t)\cdot \xi (x,t)  + \mu \rho(x,t) \div \xi(x,t)  + p(x,t) \div \xi(x,t) \, dx .
$$
for all $\xi \in C^\infty(\overline \Omega \times (0,T))$ with $\xi\cdot n=0$ on $\pa\Omega$, which gives \eqref{eq:weakEE}. 

We complete the proof as in \cite{MRS} by noticing that the function
$$ \Theta: (\mu,F)\mapsto 
\begin{cases}
\displaystyle \int_0^T\int_\Omega \frac {|F|^2}{\mu}  & \mbox{  if } F \ll \mu \mbox{ a.e. } t\in [0,T]\\
+\infty & \mbox{ otherwise}
\end{cases}
$$
is lower semi-continuous for the weak convergence of measure.
Together with the uniform bound $\Theta (\rho^{\tau_n} ,E^{\tau_n} ) = \int_0^T\int_\Omega \rho ^{\tau_n} |v^{\tau_n}|^2 \leq C$, it implies that $E$ is absolutely continuous with respect to $\rho$ and that there exists $v(t,\cdot) \in L^2 (d\rho(t))$ such that $E = \rho v$.
 Furthermore, we get
$$
\int_0^T \int_\Omega |v|^2 d\rho = \int_0^T \int_\Omega  \frac {|E|^2}{\rho} \leq 
\liminf  \int_0^T \int_\Omega  \frac {|E^{\tau_n}|^2}{\rho^{\tau_n}}
=
\liminf 
\int_0^T \int_\Omega |v^{\tau_n}|^2 d\rho^{\tau_n}
$$
which, together with \eqref{eq:disstau}, implies \eqref{eq:energy}.
\end{proof}

To complete the proof of Theorem \ref{thm:conv}, it only remains to show that $p\in P(\rho)$, which follows from  the following lemma:
\begin{lemma}\label{lem:pH}
The limiting pressure $p(x,t)$ satisfies $p\geq 0$ a.e. and 
$$
\int_\Omega p(x,t)(1-\rho(x,t)) \, dx = 0 \mbox{ a.e.  } t>0.
 $$
 In particular, $p(x,t)\in H^1_{\rho(t)}$ a.e. $t>0$.
 \end{lemma}
Since we have 
$\int_\Omega p^\tau(x,t)(1-\rho^\tau(x,t)) \, dx = 0$,
the difficulty is in passing to the limit in the product $p^\tau \rho^\tau$ (both functions only converge weakly). 
The crucial observation is that $p^\tau$ is bounded in $L^2(0,T;H^1(\Omega))$ while $\rho^\tau$ is bounded in $C^{1/2}(0,T;H^{-1}(\Omega))$ so that some compensated compactness type result can be used to pass to the limit.
We refer the reader to \cite{MRS} for a detailed proof of this lemma.

\section{Obstacle problem and strong convergence of the pressure}

We now want to prove Proposition \ref{prop:obs} (and its consequences), that is the fact that the limiting pressure $p$ is the unique solution of an obstacle problem (depending on $\rho(t_0)$).
This is not a surprising result (in fact, following \cite{MRS} one can show that for a.e. $t>0$ $p(t)$ solves an elliptic equation), but  we believe that it is interesting that, as we did in \cite{GKM} using the porous media approximation, there is a simple variational proof of this fact using the JKO scheme.
The key to the proof of Proposition \ref{prop:obs} is   Lemma \ref{lem:pressureineq}.


\begin{proof}[Proof of Proposition \ref{prop:obs}]
Equation \eqref{eq:vzeta}  together with \eqref{eq:EL5} give
$$
 \int_\Omega (-\na p^{n}  + \rho^n \na \phi^n -\mu \na \rho^n) \cdot \na \zeta\, dx 
 = \frac 1 \tau   \int_\Omega [\rho^n-\rho^{n-1}]\, \zeta  \, dx + \mathcal O\left(\frac 1 \tau  W_2^2(\rho^n,\rho^{n-1})\right)
$$
while Lemma \ref{lem:pressureineq} (in fact \eqref{eq:pressureobst} suffices here) implies
$$ \int_\Omega  ( \na p^n -\na\phi^n)\cdot \na p^{n}  \, dx  \leq 0.$$
Adding these inequalities, we find
\begin{align*} 
\int_\Omega ( \na p^n - \na \phi^n  ) \cdot \na ( p^n -\zeta)    \, dx 
& \leq
  \frac 1 \tau   \int [\rho^n-\rho^{n-1}]\, \zeta (x) \, dx + \int_\Omega(1- \rho^n) \na \phi^n \cdot \na \zeta\, dx \\
& \qquad\qquad  +  \mu \int_\Omega \na \rho^n \cdot\na \zeta\, dx + \mathcal O\left(\frac 1 \tau W_2^2 (\rho^n,\rho^{n-1})\right).
\end{align*}
We deduce
\begin{align*}
\int_{t_0}^{t_0+\delta}  \int  ( \na p^\tau  -\na \phi^\tau ) \cdot \na (p^\tau -  \zeta)  \, dx
& \leq  \int  [\rho^\tau(t_0+\delta)-\rho^\tau(t_0) ] \, \zeta (x) \, dx +\int_{t_0}^{t_0+\delta}  \int (1-\rho^\tau)\, \na \phi  ^\tau\cdot \na \zeta \, dx 
  \\
& +\mu \int_{t_0}^{t_0+\delta}  \int \na \rho^\tau \cdot \na \zeta \, dx+\mathcal O\left(\sum W_2^2 (\rho^n,\rho^{n-1})\right) + \mathcal O(\tau).
\end{align*}
We can now pass to the limit $\tau\to 0$:
For the left hand side, we use the weak convergence of $\na p^\tau$ in $L^2$, the lower semicontinuity of the $L^2$ norm and the fact that $\na \phi^\tau$ convergence strongly in $L^2$ (Lemma \ref{lem:phibd}).
For the right hand side, we use the convergence of $\int \rho^\tau(t)\zeta$ (which follows from Lemma \ref{lem:tech} and the uniform convergence of $\rho^\tau(t)$ with respect to $W_2$), 
the strong convergence of $\na \phi^\tau$ in $L^2$ and Lemma \ref{lem:unifestimates1} (i).
We deduce:
\begin{align*}
& \int_{t_0}^{t_0+\delta}  \int ( \na p -\na \phi ) \cdot \na (p - \zeta)   \, dx\, dt \\
&\qquad\qquad  \leq  \int  [\rho(t_0+\delta)-\rho(t_0) ] \, \zeta (x) \, dx  +\int_{t_0}^{t_0+\delta}  \int (1-\rho)\, \na \phi  \cdot \na \zeta \, dx\, dt +\mu \int_{t_0}^{t_0+\delta}  \int \na \rho  \cdot \na \zeta \, dx
\end{align*}
In particular, if we take $\zeta\in H^1_{\rho(t_0)}$, we find
\begin{align*}
& \int_{t_0}^{t_0+\delta}  \int ( \na p -\na \phi ) \cdot \na (p - \zeta)  \, dx\, dt \\
&\qquad\qquad \leq  \int  [\rho(t_0+\delta)-1 ] \, \zeta (x) \, dx + \int_{t_0}^{t_0+\delta}  \int (\rho(t_0)-\rho(t))\, \na \phi  \cdot \na \zeta \, dx \, d t+\mu \int_{t_0}^{t_0+\delta}  \int \na \rho  \cdot \na \zeta \, dx\\
&\qquad\qquad \leq 
\int_{t_0}^{t_0+\delta} W_2(\rho(t_0),\rho(t)) \, dt \| \na \phi  \cdot \na \zeta\|_{L^\infty(0,T;H^1(\Omega))} +\mu \int_{t_0}^{t_0+\delta}  \int \na \rho  \cdot \na \zeta \, dx
\end{align*}
where we use the fact that $\rho(t_0+\delta)\leq 1$ a.e. and Lemma \ref{lem:tech}.
Since $\rho\in C^{1/2}(0,T;\mathcal P(\Omega))$, we obtain
\begin{align*}
\int_{t_0}^{t_0+\delta}  \int( \na p -\na \phi ) \cdot \na (p - \zeta)  \, dx\, dt
& \leq 
C \| \na \phi  \cdot \na \zeta\|_{L^\infty(0,T;H^1(\Omega))} \delta^{3/2}+\mu \int_{t_0}^{t_0+\delta}  \int \na \rho  \cdot \na \zeta \, dx
\end{align*}
and  it remains to divide by $\delta$ and pass to the limit to get 
\begin{equation}\label{eq:hjkfd}
  \int (\na p(t_0) -\na  \phi(t_0)) \cdot  \na (p(t_0) - \zeta) \, dx\leq \mu  \int \na \rho (t_0)  \cdot \na \zeta \, dx =0\end{equation}
at all the Lebesgue points of $p$ (that is a.e. $t_0>0$). 

Since we already proved that $p(t_0)\in H^1_{\rho(t_0)}$ a.e. $t_0>0$ (see  Lemma \ref{lem:pH}),
this proves that $p(t_0) $ is the unique solution of \eqref{eq:obst0} a.e. $t_0>0$.

\medskip

It remains to prove that $\na p^\tau$ converges strongly in $L^2$.
We note that \eqref{eq:pressureobst} gives
$$
\int_0^T \int_\Omega |\na p^\tau|^2\, dx \, dt\leq \int_0^T \int_\Omega \na p^\tau\cdot \na \phi^\tau \, dx \, dt
$$
Since  $\na p^{\tau_n} $ converges to $\na p$   weakly in $L^2(\Omega_T)$
and $\na\phi^{\tau_n}$ converges strongly in $L^2$ to $\na \phi$, we deduce
\begin{equation}\label{eq:ineqliminf}
\int_0^T \int_\Omega |\na p|^2\, dx \, dt \leq \liminf \int_0^T \int_\Omega |\na p^{\tau_n}|^2\, dx \, dt
\leq \int_0^T \int_\Omega \na p \cdot \na \phi  \, dx \, dt.
\end{equation}
However, taking $\zeta(x)=2p(x,t_0)$ in \eqref{eq:hjkfd} we find
$$
- \int (\na p(t_0) -\na  \phi(t_0)) \cdot  \na p(t_0)  \, dx  \leq 0\qquad \mbox{ a.e. } t_0>0$$
which implies
$$ \int_0^T \int_\Omega \na p \cdot \na \phi  \, dx \, dt \leq  \int_0^T \int_\Omega |\na p|^2\, dx \, dt$$
It follows that the inequalities in \eqref{eq:ineqliminf} are in fact equalities  and thus that $\na p^{\tau_n} $ converges to $\na p$   strongly in $L^2(\Omega_T)$.
\end{proof}

\medskip
We can now prove the solution of the discrete time obstacle problem \eqref{eq:obst0tau} converges to the same limit as $p^\tau$:
\begin{proof}[Proof of Proposition \ref{prop:obs2}]
Equation \eqref{eq:obst0tau}
 implies in particular
$$
 \int_\Omega | \na q^\tau(t) |^2\, dx \leq  \int_\Omega \na \phi^\tau(t)  \cdot \na q^\tau   (t) dx 
$$
so that $ \na q^\tau$ is bounded in $L^2(\Omega_T)$. Note also that $q^\tau$ is bounded in $L^\infty(\Omega_T)$ by the maximum principle applied to the obstacle problem.
We deduce that $q^{\tau_n}$ and $\na q^{\tau_n}$ converge weakly (up to a subsequence) to $\bar q $ and $\na \bar q$ with
\begin{equation}\label{bhjk|}
\int_0^T \int_\Omega | \na \bar q |^2\, dx \, dt \leq \liminf \int_0^T \int_\Omega | \na q^{\tau_n} |^2\, dx \leq \int_0^T \int_\Omega \na \phi   \cdot \na \bar q   \,  dx \, dt.
\end{equation}
Furthermore, we can proceed as in the proof Lemma \ref{lem:pH} to show that $\bar q(t_0) \in H^1_{\rho(t_0)}$ a.e. $t_0>0$.
Next, taking $\zeta = p^\tau$ in \eqref{eq:obst0tau} and passing to the limit (using the strong convergence of $\na p^\tau$ and the lower-semicontinuity of the $L^2$ norm), we get
$$
\int_0^T \int_\Omega [ \na \bar q(x,t)- \na \phi (x,t)] \cdot  [\na \bar q(x,t)-\na p(x,t)]  dx\, dt \leq 0.
 $$
Furthermore, since $\bar q(t_0) \in H^1_{\rho(t_0)}$ a.e. $t_0>0$, we can take $\zeta = \bar q$ in 
\eqref{eq:obst0} to get
$$
\int_0^T \int_\Omega [ \na p(x,t)- \na \phi (x,t)] \cdot  [\na p(x,t) - \na \bar q(x,t)]  dx\, dt \leq 0.
 $$
Together, these inequalities imply
$$
\int_0^T \int_\Omega [ \na \bar q(x,t)-\na p(x,t)] ^2 \, dx\, dt \leq  0
 $$
and so $\bar q=p$ a.e..
Finally, this implies that $\int_0^T \int_\Omega | \na \bar q |^2\, dx \, dt= \int_0^T \int_\Omega \na \phi   \cdot \na\bar  q   \,  dx \, dt$ and \eqref{bhjk|} yields the strong convergence of $q^\tau$.
\end{proof}

\medskip

The proof of Corollary \ref{cor:pressure} is classical and is given here for the sake of completeness:

\begin{proof}[  Proof of Corollary \ref{cor:pressure}]
We recall that the pressure $p$ has been redefined (on a set of measure zero) so that it coincides with the solution of the obstacle problem \eqref{eq:obst0} for all $t>0$.

Given a test function $u\in\mathcal D(\Omega_T)$ such that $\| u\|_{L^\infty} \leq 1$ the functions $\zeta_\pm= p(1\pm u)$ satisfies $\zeta_\pm(t) \in H^1_{\rho(t)}$ for all $t>0$ and so \eqref{eq:obst0} implies
$$
\pm \int_\Omega (-\na p(t) + \na \phi(t) )\cdot \na (p(t) u(t) )\, dx \leq 0 \quad\mbox{ for all } t>0 
$$
and the complementarity condition \eqref{eq:compcond} follows.

Next, we rewrite \eqref{eq:obst0}(for a fixed $t>0$) as
\begin{equation}\label{eq:obst32} 
\int_\Omega (\na p-\na \phi)\cdot (\na p -\na \zeta) \, dx\leq 0 \qquad \forall \zeta\in H^1_{\rho(t)}.
\end{equation}
Taking $\zeta=0$ and $\zeta =2p(t)$ in this inequality immediately gives
$$ \int_\Omega (\na p-\na \phi)\cdot \na p \, dx =0 $$
and using the fact that $\rho \na p=\na p$ since $p\in H^1_{\rho(t)}$, we get the orthogonality condition
\begin{equation}\label{eq:ortho} 
\int_\Omega \na p(t) \cdot v (t)\, dx=0
\end{equation}
where the velocity $v(t)$ satisfies $\rho(t) v(t) = \rho(t)\na \phi(t) - \na p(t)$.
Furthermore, \eqref{eq:obst32} now implies
$$
\int_\Omega (\na \phi -\na p) \cdot  \na \zeta \, dx\leq 0 \qquad \forall \zeta\in H^1_{\rho(t)}
$$
and using the fact that $\rho  \na \zeta= \na \zeta$ for  $\zeta\in H^1_{\rho(t)}$, 
we deduce
$$
\int_\Omega v(t) \cdot  \na \zeta \, dx\leq 0 \qquad \forall \zeta\in H^1_{\rho(t)}
$$
that is $v(t) \in C(\rho(t))$ for all $t>0$. Together with \eqref{eq:ortho} this says that $v(t) = P_{C(\rho(t))} (\na \phi(t))$.
\end{proof}

\medskip

We end this section with a  lemma which will be useful in the proof of Theorem \ref{thm:charac}:
\begin{lemma}\label{lem:pos}
Let $p(t)$ be the unique solution of the obstacle problem \eqref{eq:obst0} and denote $\mathcal P(t)=\{p(t)>0\}$.
Then $\mu_t  = \Delta p -  \Delta \phi \chi_{\mathcal P(t)}$ is a non-negative radon measure supported on $\pa \mathcal P(t)$.
\end{lemma}

\begin{proof}
First, we recall that  the solution  of the obstacle problem \eqref{eq:obst0} satisfies  
$p\in C^{1,1}_{loc}(O(t))$ with $\mathcal O(t) := \mathrm{Int} (\{\rho(t)=1\})$
and
\begin{equation}\label{eq:pob}
 \Delta p =\Delta \phi\chi_{\{p>0\}} \mbox{ in } \mathcal O(t).
\end{equation}

Next, we  show that $\supp \mu_t \subset \pa \mathcal{P}(t)\setminus \mathcal{O}(t)$:
For all smooth test functions $\zeta\in \mathcal D(\Omega)$, by definition of $\mu_t$ we have
$$ \mu_t(\zeta) = \int_\Omega (-\na p \cdot \na \zeta -\Delta \phi  \chi_{\mathcal{P}(t)} \zeta)\, dx.$$
Clearly, if $\zeta$ is supported in $\{ p(\cdot,t)=0\}$, the fact that $p\in H^1(\Omega)$ implies that $\na p=0$ a.e. in $\{p=0\}$ and thus $ \mu_t(\zeta)=0$.
And if $\zeta$ is supported in $\mathcal O(t)$, \eqref{eq:pob} implies
$$ \mu_t(\zeta) =0.$$
Since $\mathcal O(t)$ is an open set, we deduce that
$$ \supp(\mu_t) \cap \mathrm{Int} (\{ p(t)=0\}) =\emptyset, \qquad \supp(\mu_t) \cap \mathcal O(t)=\emptyset.
$$

On the  other hand, note that $\mathrm{Int}(\mathcal{P}(t))\subset \mathcal O(t)$. Indeed if $p(t)>0$ in $B_\delta(x_0)$, then $1-\rho(t)=0$ a.e. in $B_\delta(x_0)$. It follows that $x_0\in \mathcal O(t)$. Thus we can conclude that 
 $\mu_t$ is supported in  $\pa\mathcal{P}(t)\setminus \mathcal O(t)$.

\medskip

Next we show that $\mu_t$ is nonnegative. 
Define the function
$$Q_\delta(s):=\begin{cases}
\frac{s}{\delta} & \mbox{ if } s\in[0,\delta];\\
1 & \mbox{ if } s\geq \delta.
\end{cases}
$$
For any test function $\zeta \in \mathcal D(\Omega)$ satisfying $0\leq \zeta(x)\leq 1$, we write
\begin{align*}
\mu_t(\zeta) 
& = \int_\Omega - \na p \cdot \na \zeta -\Delta \phi  \chi_{\mathcal{P}(t)} \zeta\, dx\\
& = \int_\Omega -\na p \cdot \na (\zeta Q_\delta(p)) -\Delta \phi  \zeta Q_\delta(p) \, dx
 + \langle \Delta p, \zeta (1-Q_\delta(p)) \rangle  - \int_\Omega \Delta \phi \chi_{\mathcal{P}(t)}\zeta(1-Q_\delta(p)) \, dx.
\end{align*}
Using  \eqref{eq:obst0} with test function $p+ \zeta Q_\delta(p)$ (which is in $H^1_{\rho(t)}$) we see that the first integral is non-negative. Next note that
 $$
 \langle \Delta p, \zeta (1-Q_\delta(p)) \rangle  =  \int \nabla p \cdot \nabla  \varphi(Q_\delta(p)-1) + \nabla p \cdot \varphi Q'_{\delta}(p) \nabla p.
 $$
 The second term in above equality is nonnegative since $Q_{\delta}$ is increasing.
 For the first term, we note that $ \nabla  \varphi(Q_\delta(p)-1)$ converges a.e. to $\nabla  \varphi \chi_{\{p=0\}}$.  Lebesgue dominated convergence theorem implies that it converges in $L^2$ and thus the first term converges to
zero since $\na p =0 $ a.e. in $\{ p=0\}$. 
 
 \medskip

  Thus
$$
\mu_t(\zeta) \geq - \int_\Omega \Delta \phi \chi_{\mathcal{P}(t)}\zeta(1-Q_\delta(p)) \, dx.$$
Finally, we have $ \chi_{\mathcal{P}(t)} (1-Q_\delta(p)) \to 0$ a.e. in $\Omega$ when $\delta\to0$. Sending $\delta\to 0$ and using Lebesgue dominated convergence theorem, we can conclude
that $\mu_t(\zeta) \geq 0$ and the result follows.

\end{proof}

\section{Uniqueness}\label{sec:unique}
This section is devoted to the proof of Proposition \ref{prop:unique}. The proof uses ideas first introduced in \cite{PQV}, but which must be carefully adapted due to the lack of appropriate regularity of the potential $\phi$ (namely, the fact that we do not have $\phi\in L^\infty(0,T;W^{2,\infty}(\Omega))$).
We consider two functions $(\rho_1,p_1)$ and $(\rho_2,p_2)$, solutions of
$$
\begin{cases}
\pa_t \rho -\mu \Delta\rho + \div (\rho \na \phi-\na p) =0 , \qquad p\in P(\rho), \\
\phi (x) = \int_\Omega G(x,y) \rho(y)\, dy \\
\end{cases}
$$
in the sense of Definition \ref{def:weak} and with same initial data $\rho_{in}$. 
We have in particular
$$
\int_\Omega \rho_{in} (x) \psi(x,0)\, dx + \int_0^\infty \int_\Omega (\rho_i \, \pa_t\psi +  \rho_i \na \phi_i \cdot \na \psi+ (\mu\rho_i+p_i)\, \Delta \psi ) \, dx \, dt= 0 , \qquad i=1,\;2
$$
for any function $\psi\in C^\infty_c(\overline \Omega \times [0,\infty))$ satisfying $\na \psi \cdot n=0$ on $\pa\Omega$.
To simplify the notation, we will write $G* \rho_i$ for $\phi_i=\int_\Omega G(x,y) \rho_i(y)\, dy$ below, even though it is not really  a convolution.

For a test function $\psi\in C^\infty_c(\overline \Omega_T) $ satisfying $\psi(T)=0$ and  $\na \psi \cdot n=0$ on $\pa\Omega$, we can write
\begin{equation}\label{eq:psi0}
\int_0^T \int_\Omega
(\rho_1-\rho_2+p_1-p_2) \Big(
A\pa_t \psi + (\mu A+B)\Delta \psi +A \na \phi_1\cdot \na \psi + A\na G*(\rho_2 \na \psi) \Big) \, dx\, dt=0
\end{equation}
where
\begin{align*}
0\leq A= \frac{\rho_1-\rho_2}{\rho_1-\rho_2+p_1-p_2}\leq 1\\
0\leq B= \frac{p_1-p_2}{\rho_1-\rho_2+p_1-p_2}\leq 1
\end{align*}
(with the convention that $A=0$ whenever $\rho_1-\rho_2=0$ and $B=0$ whenever $p_1-p_2=0$).
To get this equality, we used in particular the fact that
\begin{align*}
\int_0^T \int_\Omega \rho_2 \na (\phi_1-\phi_2) \cdot \na\psi\, dx \, dt 
& =\int_0^T \int_\Omega \int_\Omega   \rho_2(x) \na G (x,y) (\rho_1(y)-\rho_2(y)) \cdot \na\psi(x)\, dx \,dy \, dt \\
& =\int_0^T \int_\Omega  \int_\Omega (\rho_1(x)-\rho_2(x))   \na G (x,y)  \rho_2(y) \na\psi(y)\, dx \,dy\, dt .
\end{align*}
In order to prove uniqueness, we want to solve the dual equation
$$
\begin{cases}
A\pa_t \psi + (\mu A+B)\Delta \psi +A \na \phi_1\cdot \na \psi + A\na G*(\rho_2 \na \psi) = A h  \quad\hbox{ in }  \Omega_T;\\
\psi(x,T)=0 \,\,\mbox{ on } \Omega, \qquad \na \psi \cdot n=0 \,\,\mbox{ on } \pa\Omega.
\end{cases}
$$
for any reasonable test function $h$.
Since this equation is not uniformly parabolic, we first need to regularize  $A$, $B$ and $\phi_1$:
As in \cite{PQV}, we first consider sequences $A_n$, $B_n$ of smooth bounded functions such that
\begin{align*}
\| A-A_n\|_{L^2} \leq C/n, \qquad 1/n\leq A_n\leq 1 , \\
\|B-B_n\|_{L^2} \leq C/n, \qquad 1/n\leq B_n\leq 1.
\end{align*}

In addition, since $D^2\phi_1\notin L^\infty$, we approximate $\phi_1$ by a appropriate sequence of function $\phi_{1,n}$. More precisely, we will use the following lemma:
\begin{lemma}\label{lem:phin}
There exists $\lambda,\gamma>0$ and a sequence $\phi_{1,n}$ such that
\begin{equation}\label{eq:phiapprox}
\| \na \phi_{1,n} - \na \phi_1 \|_{L^2(\Omega)} \leq \gamma /\sqrt n , \qquad  \| D^2 \phi_{1,n}\|_{L^{\infty} (\Omega)} \leq \lambda \ln( n) 
\end{equation}
\end{lemma}

Postponing the proof of this lemma to the end of this section, we can now solve the approximate equation
\begin{equation}\label{eq:psin}
\begin{cases}
\pa_t \psi_n + \left(\mu+\frac{B_n}{A_n}\right)\Delta \psi_n + \na \phi_{1,n}\cdot \na \psi_n + \na G*(\rho_2 \na \psi_n) =  h \quad \hbox{ in } \Omega_T;\\
\psi_n(x,T)=0  \,\,\mbox{ on } \Omega, \qquad \na \psi_n \cdot n=0 \,\,\mbox{ on } \pa\Omega.
\end{cases}
\end{equation}
since the diffusion coefficient satisfies in particular $\frac 1 n \leq \frac{B_n}{A_n}\leq n$.
In order to pass to the limit $n\to\infty$, we will need the following a priori estimates on  $\psi_n$:
\begin{lemma}\label{lem:bdpsin}
There exists a constant $C$, depending on $h$, but independent of $n$ such that
$$ \| \psi_n\|_{L^\infty(\Omega_T)}\leq C T$$ and 
$$\sup_{t\in(0,T)}\int_\Omega |\na \psi_n(t)|^2\, dx \leq e^{CT} n^{\lambda T}, \quad 
 \int_0^T\int_\Omega  \frac{B_n}{A_n}|\Delta \psi_n|^2 \, dx\,dt \leq e^{CT} n^{\lambda T}.
$$
\end{lemma}
\begin{proof}
The first bound follows from the comparison principle (using $\| h\|_{L^\infty}\times (T-t)$ as barrier).

For the second bound, we multiply by $-\Delta \psi_n$ to get:
\begin{align*}
\frac{d}{dt} \int_\Omega |\na \psi_n|^2\, dx 
& = \int _\Omega  \left(\mu+\frac{B_n}{A_n}\right)|\Delta \psi_n|^2\,dx +\int _\Omega (\na \phi_{1,n}\cdot \na \psi_n + \na G*(\rho_2 \na \psi_n))\Delta \psi_n\, dx
-\int_\Omega \Delta h \psi_n\, dx
\end{align*}
and we need to bound the second integral in the right hand side.
First, we recall that the term $\na G*(\rho_2 \na \psi_n))\Delta \psi_n$ should be written as $ \div U  \Delta \psi_n\, dx$ where $U$ solves \eqref{eq:phi0} with $\rho_2 \na \psi_n$ in the right hand side.
We thus have (since $\na \psi_n\cdot n=0$ on $\pa\Omega$ and using Calderon-Zygmund estimates):
\begin{align*} \int_\Omega \na G*(\rho_2 \na \psi_n)\Delta \psi_n\, dx = 
\int_\Omega \na \div U \cdot \na \psi_n \, dx & \leq \| D^2 U \|_{L^2(\Omega)} \|\na \psi_n\|_{L^2(\Omega)}\\
&  \leq C \| \rho_2 \na \psi_n \|_{L^2(\Omega)} \|\na \psi_n\|_{L^2(\Omega)}
\leq C  \|\na \psi_n\|_{L^2(\Omega)}^2
\end{align*}
Next, we can write
$$ \int _\Omega (\na \phi_{1,n}\cdot \na \psi_n )\Delta \psi_n\, dx = \int_\Omega ( -D^2 \phi_{1,n} : \na \psi_n\otimes\na\psi_n+ \Delta\phi_{1,n} \frac{|\na \psi|^2}{2})\, dx.$$
Since $\| D^2 \phi_{1,n}\| \leq \lambda \ln( n)$ (see Lemma \ref{lem:phin}),  we deduce
\begin{align*}
\frac{d}{dt} \int |\na \psi_n|^2\, dx 
& \geq  \int \left(\mu+\frac{B_n}{A_n}\right)|\Delta \psi_n|^2 - C (1+\lambda \ln(n)) \int | \na \psi_n|^2-C.
\end{align*}
Gronwall's inequality, together with the fact that $\int |\na \psi_n|^2 (T)\, dx= 0$ implies
\begin{align*}
\sup_{t\in(0,T)}\int |\na \psi_n(t)|^2\, dx + 
\int_0^T  \int \left(\mu+\frac{B_n}{A_n}\right)|\Delta \psi_n|^2 \, dx\,dt
& \leq \frac{e^{ C (1+\lambda \ln(n))T}-1 }{  (1+\lambda \ln(n))}\leq e^{CT} n^{\lambda T},
\end{align*}
which completes the proof of Lemma \ref{lem:bdpsin}.
\end{proof}

We can now complete the proof of Proposition \ref{prop:unique}:

\begin{proof}[Proof of Proposition \ref{prop:unique}]
First, we fix $T=1/(2\lambda)$. Equation \eqref{eq:psin} yields
\begin{align*}
0
& =\int_0^T \int_\Omega
(\rho_1-\rho_2+p_1-p_2) \Big(
A\pa_t \psi_n + (\mu A+B)\Delta \psi_n +A \na \phi_1\cdot \na \psi_n + A\na G*(\rho_2 \na \psi_n) \Big) \, dx\, dt\\
& =\int_0^T \int_\Omega
(\rho_1-\rho_2+p_1-p_2) \Big(Ah +B\Delta \psi_n -  A \frac{B_n}{A_n}\Delta \psi_n 
+ A (\na \phi_1- \na \phi_{1,n})\cdot \na \psi_n 
\Big) \, dx\, dt
\end{align*}
and so, using \eqref{eq:psi0}, we can write:
\begin{align*}
\int_0^T \int_\Omega
(\rho_1-\rho_2) h
& =\int_0^T \int_\Omega
(\rho_1-\rho_2+p_1-p_2) \Big(A \frac{B_n}{A_n}-B \Big) \Delta \psi_n\, dx\, dt\\
& =\int_0^T \int_\Omega
(\rho_1-\rho_2+p_1-p_2) \frac{B_n}{A_n} \Big(A -A_n \Big) \Delta \psi_n\, dx\, dt \\
& \qquad +\int_0^T \int_\Omega
(\rho_1-\rho_2+p_1-p_2)  \Big(B_n-B \Big) \Delta \psi_n\, dx\, dt \\
& \qquad + \int_0^T \int_\Omega(\rho_1-\rho_2+p_1-p_2)  \Big( A (\na \phi_1- \na \phi_{1,n})\cdot \na \psi_n 
\Big) \, dx\, dt.
\end{align*}
We now show that the three terms in the right hand side go to zero as $n\to\infty$:
First, we have
\begin{align*}
\left| \int_0^T \int_\Omega
(\rho_1-\rho_2+p_1-p_2) \frac{B_n}{A_n} \Big(A -A_n \Big) \Delta \psi_n\, dx\, dt \right|
&  \leq C \left( \int_\Omega \frac{B_n}{A_n} \Big(A -A_n \Big)^2 \, dx\, dt \right)^{1/2}  e^{CT/2} n^{\lambda T/2}\\
&  \leq Cn^{1/2} \left( \int_\Omega   \Big(A -A_n \Big)^2 \, dx\, dt \right)^{1/2}  e^{CT/2} n^{\lambda T/2}\\
&  \leq Cn^{-1/2}  e^{CT/2} n^{\lambda T/2}
\end{align*}
which goes to zero since  $T <1/\lambda $.
Similarly,
\begin{align*}
\left| \int_0^T \int_\Omega
(\rho_1-\rho_2+p_1-p_2)  \Big(B_n-B \Big) \Delta \psi_n\, dx\, dt\right|
& \leq 
\left( \int_0^T \int_\Omega
 \frac{A_n}{B_n} \Big(B_n-B \Big)^2 \, dx\, dt\right)^{1/2}e^{CT/2} n^{\lambda T/2}\\
 \leq Cn^{-1/2}  e^{CT/2} n^{\lambda T/2} \to 0
 \end{align*}
and 
\begin{align*}
\left|
\int_0^T \int_\Omega(\rho_1-\rho_2+p_1-p_2)  \Big( A (\na \phi_1- \na \phi_{1,n})\cdot \na \psi_n 
\Big) \, dx\, dt
\right|
& \leq
\| \na \phi_1- \na \phi_{1,n} \| _{L^2}\| \na \psi_n \|_{L^2}\\
& \leq \gamma n^{-1/2}   e^{CT/2} n^{\lambda T/2} \to 0.
\end{align*}

We have thus showed that 
$\int_0^T \int_\Omega (\rho_1-\rho_2) h\, dx\,dt=0$
for all test function $h$ and thus $\rho_1(x,t)=\rho_2(x,t)$ a.e. for $x\in \Omega$ and $t\in [0, T_0]$ with $T_0=1/(2\lambda)$.
This short time uniqueness result easily yield uniqueness for all time by iteration over the time interval.


\end{proof}

\begin{proof}[Proof of Lemma \ref{lem:phin}]
The proof makes use of the following bounds on the kernel $G$:
\begin{equation}\label{green_bd}
|\pa_{i} G(x,y)| \leq C |x-y|^{-d+1}\, , \qquad |\pa_{ij} G(x,y)| \leq C |x-y|^{-d}\qquad \forall (x,y)\in \Omega\times\Omega, \; x\neq y,
\end{equation}
which  follow from the uniform bounds established in \cite{ChoiKim}. 
\medskip

We define $\rho_n (x) = \int_{\Omega} \eta_n (x-y) \rho(y)\, dy$ where $\eta_n(x) = \frac{n^d}{\omega_n} \chi_{B_{1/n}}(x)$ and $\phi_n(x)= \int_\Omega G(x,y)\rho_n(y)\, dy$.
Since $\rho\leq1$, we have
\begin{align*}
|\rho_n(x)-\rho_n(x') | 
& \leq \int_{\R^d}|  \eta_n (x-y)- \eta_n (x'-y)|\, dy \\
& \leq 
\begin{cases}
C n |x-x'|  & \mbox{ if } |x-x'|\leq 1/n \\
1 & \mbox{ otherwise }
\end{cases}
\end{align*}
and so
\begin{align*}
\frac{|\rho_n(x)-\rho_n(x') | }{|x-x'|^s} 
&\leq 
\begin{cases}
C n |x-x'|^{1-s}  & \mbox{ if } |x-x'|\leq 1/n \\
  |x-x'|^{-s}  & \mbox{ otherwise }
\end{cases} \\
& \leq Cn ^s
\end{align*}

Thus for any $x\in \Omega$ we write
$$  \pa_{ij}\phi_n(x) = \int_\Omega \pa_{ij} G(x,y) \rho_n(y) \, dy =  \int_\Omega \pa_{ij} G(x,y) [\rho_n(y)-\rho_n(x)]\, dy + \int_\Omega \pa_{ij} G(x,y)  \, dy\rho_n(x)$$
where
$  \int_\Omega \pa_{ij} G(x,y)  \, dy = \pa_{ij} v(x)$ with $v$ solution of
$$
\begin{cases}
v- \Delta v = 1 & \mbox{ in } \Omega\\
\alpha v + \beta \na v\cdot n = 0 & \mbox{ on } \pa\Omega.
\end{cases}
$$
Classical Shauder's estimates give $\sup_{x\in\Omega}\pa_{ij} v (x) \leq C$
and so
\begin{align*}
| \pa_{ij}\phi_n(x) |
& = \left| \int_\Omega \pa_{ij} G(x,y) [\rho_n(y)-\rho_n(x)]\, dy\right| + | \pa_{ij} v(x) \rho_n(x)|\\
& \leq C n^s
 \int_\Omega \left|  \pa_{ij} G(x,y)\right| |y-x|^s \, dy + | \pa_{ij} v(x) |\\
 & \leq C n^s
 \int_\Omega |y-x|^{s-d}  \, dy +C  \\
&\leq   \frac{C}{s} n^s + C,
\end{align*}
where the constant $C$ depends on $\Omega$, but not on $s$. It remains to take $s = \frac{1}{\ln n}$ to optimize the right hand side,
which yields 
$$ | \pa_{ij}\phi_n |\leq C \ln n \qquad \hbox{ in }  \Omega.
$$

Next, we consider the function $z_n = \phi-\phi_n$ and prove that
\begin{equation}\label{eq:zL1} 
\| z_n\|_{L^1(\Omega)} \leq \frac C n.
\end{equation}
To get this estimate, we denote $\Omega_{1/n} = \{ x\in \Omega \,;\, \mathrm{dist}(x,\pa\Omega) \geq 1/n\}$ and we write:
\begin{align*}
\phi(x)-\phi_n(x) 
&= \int_\Omega \int_{\Omega_{1/n}} G(x,y) \eta_n(y-z) [\rho(z) -\rho(y)]\, dy\, dz   +  \int_\Omega \int_{\Omega\setminus \Omega_{1/n} } G(x,y) \eta_n(y-z) [\rho(z) -\rho(y)]\, dy\, dz\\
& = I_1(x) + I_2(x)
\end{align*}
Since 
$$|I_2(x)|\leq  \int_\Omega \int_{\Omega\setminus \Omega_{1/n} } G(x,y) \eta_n(y-z)\, dy\, dz  =   \int_{\Omega\setminus \Omega_{1/n} } G(x,y)  \, dy$$
we have
$$\int |I_2(x)|\, dx \leq \int_{\Omega\setminus \Omega_{1/n} } \int_\Omega G(x,y) \, dx \, dy\leq |\Omega\setminus \Omega_{1/n}| \leq \frac Cn.$$

Next, we note that for $y\in\Omega_{1/n}$, we have $\mathrm {Supp } \, \eta_n(y-\cdot) \subset B_{1/n}(y) \subset \Omega$ and so
\begin{align*}
I_1 (x) 
&= \int_\Omega  \int_{\Omega_{1/n}}[G(x,z)-G(x,y)] \eta_n(y-z) \rho(y)\, dy\, dz\\
&= \int_\Omega  \int_{\Omega_{1/n}}\int_0^1 \na G(x,y+t(z-y))\cdot (z-y)  \eta_n(y-z) \rho(y)\, dy\, dz.
\end{align*}
We deduce
\begin{align*}
\int_\Omega |I_1 (x)| \, dx
&\leq \int_{\Omega}  \int_{\Omega_{1/n}} \int_0^1\int_\Omega  |\na G(x,y+t(z-y)\, dx | |z-y| \eta_n(y-z)\, dy\, dz\\
&\leq \int_{\Omega}  \int_\Omega  |z-y| \eta_n(y-z)\, dy\, dz\\
& \leq   C \frac 1 n.
\end{align*}
and \eqref{eq:zL1} follows.

Finally, we remark that $z_n$ solves
$$
\begin{cases}
z_n -\Delta z_n = \rho-\rho_n & \mbox{ in } \Omega\\
  \na z_n \cdot n = 0 & \mbox{ on }\pa \Omega.
\end{cases}
$$
Multiplying by $z_n$ and integrating, we deduce the following bound: 
$$ \int_\Omega  |\na z_n|^2\, dx\leq \int_\Omega  (\rho-\rho_n) z_n \, dx \leq\int_\Omega | z_n| \, dx  \leq \frac C n$$
which gives the first bound in \eqref{eq:phiapprox} and conclude the proof.

\end{proof}

\section{Characteristic functions: Proof of Theorem \ref{thm:charac}}

\begin{proof}[Proof of Theorem \ref{thm:charac}]
Since we now assume that $\mu=0$,  $\rho$ solves (in the sense of distribution)
$$
\pa_t \rho + \div( \rho \na \phi ) = \Delta p .
$$
With the notations of Lemma \ref{lem:pos}, we can also write this equation as (since $\rho=1$ a.e. in $\mathcal P$)
\begin{equation}\label{eq:weak1}
\pa_t \rho + \na \rho\cdot \na \phi = \mu_t - \rho (1-\chi_{\mathcal P(t)})\Delta \phi.
\end{equation}
Lemma \ref{lem:pos} together with equation \eqref{eq:phi0} then implies
$$\pa_t \rho + \na \rho\cdot \na \phi \geq   \rho (1-\chi_{\mathcal P(t)}) (\rho-\sigma \phi) 
$$
and since $\sigma\phi\leq 1$, we get (since $\rho=1$ a.e. in $\mathcal P$)
\begin{equation}\label{eq:mont}
\pa_t \rho + \na \rho\cdot \na \phi 
\geq \rho (1-\chi_{\mathcal P(t)}) (\rho-1) = - \rho  (1-\rho)
\end{equation}
(note that when working with a fixed potential $\phi$ such that $-\Delta \phi\geq 0$, we get
$\pa_t \rho + \na \rho\cdot \na \phi\geq 0$ meaning that $\rho$ is monotone increasing along the characteristic curves associated to the vector field $\na \phi$. The proof is simpler in that case).

Heuristically, the proof now goes as follows:
Given $(x_0,t_0)$, we consider the  characteristic curve
\begin{equation}\label{eq:cc} \dot{X}(t) = \na \phi(X(t),t ), \qquad X(t_0)=x_0.\end{equation}
The density along the characteristic curve, $u(t) = \rho(X(t),t)$ satisfies (thanks to \eqref{eq:mont}):
$$ 
u'(t) \geq - u(t)(1-u(t)) \geq -(1-u(t))$$
which implies 
$$ 1- u(t) \geq (1-u(t_0)) e^{-(t_0-t)} \quad \forall t<t_0.$$ 
In particular, if $u(t_0)<1$, then $u(t)<1$ for all $t<t_0$.

So given $x_0$ such that $\rho(x,t_0) <1$  in a neighborhood of $x_0$ and $X(t)$ solution of \eqref{eq:cc}, 
 \eqref{eq:mont} implies that $\rho(x,t)<1$ for all $t<t_0$ in a neighborhood of $X(t)$.
 It follows that  $p(x,t)=0$ a.e. in that same neighborhood and so $\Delta p =0$ (as a distribution) and going back to \eqref{eq:weak1}, we find
$$ 
\frac{d}{dt}\rho(X(t),t) = -\rho(X(t),t) \Delta \phi(X(t),t) \qquad \forall t\in [0,t_0].$$
Since $\rho(X(0),0)<1$, we must have $\rho(X(0),0)=0$ (this is where we use the fact that $\rho_{in}$ is a characteristic function) and thus $\rho(X(t),t)=0$ for all $t\in [0,t_0]$.

In particular $\rho(X(t_0),t_0)=\rho(x_0,t_0)=0$ (and this holds for all $t_0$ and $x_0$ such that $\rho(x,t_0) <1$  in a neighborhood of $x_0$). The result follows.

\medskip

To make this argument rigorous, we prove the following lemma:
\begin{lemma}\label{lem:w1}
Given $t_0>0$, let $\psi(x,t)$ be the solution of 
$$\pa_t \psi = \na \phi \circ \psi, \qquad \psi (x,t_0)=x.$$
Then 
$$1-\rho(\psi(x,t_1),t_1) \geq(1- \rho(x,t_0)) e^{-(t_0-t_1)} \qquad \forall t_1<t_0.$$
\end{lemma}  

\begin{proof}
Given a test function $\zeta_0(x)\geq 0$, the function
$ \zeta(x,t) = \psi(\cdot,t)\# \zeta_0$ solves
$$ \pa_t \zeta + \div(\na \phi \, \zeta) =0 ,\qquad \zeta(t_0)=\zeta_0.$$
So, taking $\zeta$ as a test function in \eqref{eq:weak1},
we find (in $\mathcal D'(0,T))$):
\begin{align}
\frac{d}{dt} \int_\Omega \rho(x,t) \zeta(x,t)\, dx 
& = \int_\Omega \pa_t \rho(x,t) \zeta(x,t)+ \rho(x,t) \pa_t \zeta(x,t), dx\nonumber  \\
& = \langle \mu_t  , \zeta(\cdot,t)\rangle - \int_\Omega \rho(x,t)  (1-\chi_{\mathcal P (t)}) \Delta \phi\zeta(x,t)\, dx \label{eq:pb}\\
& \geq  \int_\Omega \rho(x,t)  (1-\chi_{\mathcal P(t)}) (\rho-\sigma \phi)\zeta(x,t)\, dx \nonumber \\
& \geq  \int_\Omega \rho(x,t)  (1-\chi_{\mathcal P(t)}) (\rho-1)\zeta(x,t)\, dx \nonumber \\
&\geq - \int_\Omega(1- \rho(x,t) )\zeta(x,t)\, dx\nonumber
\end{align}
since 
$\frac{d}{dt} \int_\Omega \zeta(x,t)\, dx =0$ (by construction of $\zeta$), we deduce
$$\frac{d}{dt} \int_\Omega (1-\rho(x,t)) \zeta(x,t)\, dx \leq  \int_\Omega(1- \rho(x,t) )\zeta(x,t)\, dx$$
which implies (with $t_1<t_0$)
$$
 \int_\Omega (1-\rho(x,t_1)) \zeta(x,t_1)\, dx \geq  \int_\Omega (1-\rho(x,t_0)) \zeta(x,t_0)\, dx e^{-(t_0-t_1)}.
 $$
Since
$$
 \int_\Omega (1- \rho(x,t_1) ) \zeta(x,t_1)\, dx 
\int (1-\rho(x,t_1))\psi(\cdot,t_1)\# \zeta_0 \, dx = \int(1- \rho(\psi(x,t_1),t_1))\zeta_0 (x) \, dx ,
$$
the result follows.
\end{proof}

To complete the proof of the theorem, we consider $B_r(x_0) \subset \{\rho(\cdot,t)<1\}$
and a test function $\zeta_0$ supported in  $B_r(x_0)$. As in the proof above, we define $ \zeta(x,t) = \psi(\cdot,t)\# \zeta_0$.
Lemma \ref{lem:w1} implies that $\rho(x,t)<1 $, and so $p(x,t)=0$,  a.e. in $\psi(B_r(x_0),t )$ for all $t<t_0$.
Since $\supp \zeta (t) \subset \psi(B_r(x_0),t)$ we deduce that
$$  \langle \mu_t, \zeta(t)\rangle =0 \qquad \forall t<t_0.$$

We now go back to \eqref{eq:pb}: If $\zeta\geq 0$ and using the fact that $-\Delta \phi=\rho-\sigma \phi \leq 1$ we get:
$$
\frac{d}{dt} \int_\Omega \rho(x,t) \zeta(x,t)\, dx  \leq   \int_\Omega \rho(x,t) \zeta(x,t)\, dx \quad \mbox{ in } \mathcal D'((0,t_0)).
$$
Since $ \int_\Omega \rho(x,0) \zeta(x,0)\, dx=0$ (when $\rho(0)$ is a characteristic function), we deduce that $ \int_\Omega \rho(x,t_0) \zeta(x,t_0)\, dx$  and thus $\rho(x,t_0)=0$ a.e. in $B_r(x_0)$.

We have thus proved that for all $t>0$, $\rho(x,t)=0$ a.e. in the interior of $\Omega\setminus \{\rho(\cdot,t)=1\}$. The theorem follows (with $\Omega_s(t) = \{\rho(\cdot,t)=1\}$).
\end{proof}



\medskip
\medskip

\appendix

\section{Proof of Proposition  \ref{prop:minchar} (ii)}
We now prove the second part of Proposition  \ref{prop:minchar}, which claims that the minimizers of \eqref{eq:min1} are characteristic functions when $\mu=0$. This remarkable fact implies that the JKO approximation $\rho^{\tau}$ are characteristic functions for all time, regardless of whether the initial data is a characteristic function or not. 
The proof presented below first appeared in \cite{JKM} and requires only minor modifications to our framework. 
We present it here for the  reader's convenience.

\medskip

The proof relies on two remarks:
\begin{enumerate}
\item The functional $J$ is concave.
\item We can rewrite the minimization problem \eqref{eq:min1} in term of the optimal plan $\pi$ instead of the density $\rho$, thus replacing  the nonlinear term $\rho \mapsto W_2^2(\rho,\brho)$ by the linear term $  \pi\mapsto \int_{\Omega\times\Omega} |x-y|^2\, d\pi(x,y)$.
\end{enumerate}
Thanks to these two remarks, we end up with a minimization problem for a concave functional on a convex set. Minimizers must then be extremal points, which, in $K$, are characteristic functions.

\medskip

To make the second point more precise, we denote 
$$
S(\rho)= \frac {1}{2\tau} W_2^2(\rho,\brho) + J(\rho), \qquad \rho \in K
$$
and we introduce
$$ \widetilde S(\pi) = \frac {1}{2\tau} \int_{\Omega\times\Omega} |x-y|^2\, d\pi(x,y)+ \widetilde J(\pi), \qquad \pi \in \widetilde K$$
where
\[
\widetilde J (\pi) =\frac{1}{2\sigma} \int_{\Omega\times \Omega} d\pi(x,y) -\frac 1 2 \int_{\Omega^4} G(y_1,y_2)d\pi(x_2,y_2)d\pi(x_1,y_1) 
\]
and $\widetilde K$ is the set of plans in $\mathcal P(\Omega\times\Omega)$ whose first marginal (denoted $\pi^1$) is $\brho$ and second marginal (denoted $\pi^2$) is in $K$, that is
$$ \int_{\Omega\times\Omega} u(x) d\pi(x,y) = \int_\Omega u(x) d\brho(x), $$
$$ \int_{\Omega\times\Omega} v(y) d\pi(x,y) = \int_\Omega v(y) d\rho(y), \quad \mbox{ for some } \rho \in K.$$

Clearly,
for any $\pi \in \widetilde K$ with  $\pi^2=\rho$, we have $\widetilde S(\pi) \geq S(\rho)$.
Also, given $\rho\in K$, the optimal plan $\pi$ from $\brho$ to $\rho$ is in $\widetilde K$ and satisfies $\widetilde S(\pi) = S(\rho)$. 
We easily deduce that
$$ 
\min \{ S(\rho) \, ;\, \rho\in K\} = \min \{ \widetilde S(\pi) \, ;\, \pi \in \widetilde K\}
$$
and any minimizer of one problem is associated to a minimizer of the other one.

\medskip

To check the concavity of $\widetilde J$, we write
\[
D^2 \widetilde J(\pi^*)(\theta,\theta)=- \int_{\Omega^4}{G (y_1,y_2)d\theta(x_2,y_2)d\theta(x_1,y_1)}=-\int_{\Omega\times \Omega}{G (y_1,y_2)f(y_1)f(y_2)dy_1dy_2}
\]
with $f(y)=\theta^2(y)$. Then
\[
D^2 \widetilde J(\pi^*)(\theta,\theta)=- \int_{\Omega } f(y_1)\psi (y_1)dy_1
\]
with $\psi$ solving \eqref{eq:phi0} with right hand side replaced by $f$. 
As simple integration by part leads to
\begin{align}
D^2 \widetilde J(\pi^*)(\theta,\theta) 
&=- \int_{\Omega } |\psi (y_1)|^2 dy_1- \int_{\Omega } |\na \psi (y_1)|^2 dy_1- \int_{\pa \Omega }\frac{\alpha}{\beta} |\psi (y_1)|^2 \, d\H^{n-1}(y_1) \nonumber  \\
& \leq 0\label{eq:Dconc}
\end{align}
which implies  that $\widetilde J$, and thus $\widetilde S(\pi)$, is concave. 
We note that this last inequality is strict unless $\psi\equiv 0$, which requires $f = \theta^2\equiv 0$.

\medskip

As explained above, the fact that $\widetilde S(\pi)$ is strictly concave implies that it achieves its minimum in $\widetilde K$ at an extremal point, which corresponds to a characteristic function.
To make this point precise, let $\rho$ be a minimizer of \eqref{eq:min1} and $\pi$ the corresponding optimal plan between $\brho$ and $\rho$, which is then  a minimizer of $\widetilde S$ in $\widetilde K$.
We assume that $\rho$ is not a characteristic function, so that
there exists $0<\alpha<1 $ such that the set $\Omega_{\alpha} = \{y \in \Omega, \rho(y)\in (\alpha,1 - \alpha)\}$ has positive measure. 
Next, we partition $\Omega_{\alpha}$ into two sets $E_1,E_2$ of equal measure so that  there exists a measure preserving map $T: E_1 \to E_2$. We now construct an admissible perturbation $\theta$ of $\pi$ as 
\[
\theta=(\brho\otimes (\rho \circ T))\big|_ {\Omega \times E_1}-(\brho \otimes \rho)\big|_{\Omega\times E_2}
\]
i.e. for all $\psi \in C(\Omega \times \Omega)$,
\[
\int_{\Omega\times \Omega}{\psi(x,y)d\theta(x,y)}=\int_{\Omega \times E_1}{\psi(x,y)d\brho(x)d\rho(T(y))}-\int_{\Omega \times E_2}{\psi(x,y)d\brho(x)d\rho(y)}.
\]
\medskip

We can then check that for  $\delta =\min\{\alpha,1-\alpha\}$, we have
\[
\pi\pm t\theta\in \widetilde K \quad \mbox{ for all $t\in[0,\delta)$.}
\]
Indeed, $\pi\pm t \theta$ defines a non-negative measure for $t\in[0,\delta)$ and we have
\[
\int_{\Omega\times\Omega}{u(x)d\theta(x,y)}=\int_{\Omega}{u(x)d\brho(x)\int_{E_1}{d\rho(T(y))}}-\int_{\Omega}{u(x)d\brho(x)\int_{E_2}{d\rho(y)}}=0,
\]
which implies $\theta^1(x)=0$ a.e. $x\in \Omega$ (so the first marginal of $\pi\pm t \theta$ is $\brho$).
Taking $u=1$ above, we also find $\int_{\Omega\times \Omega}{d\theta(x,y)}=0$, which gives
$\int_{\Omega}{d\theta^2(x)}=0$.
In fact, we have $\theta^2(y) = \rho\circ T|_{E_1}(y) - \rho|_{E_2}(y)$ which is not zero.
\medskip


Since $\pi$ is a minimizer of $\widetilde S$ in $\widetilde K$, 
we  must  have $\widetilde S(\pi\pm t\theta)\geq \widetilde S(\pi)$. 
However \eqref{eq:Dconc} and the fact that $\theta^2\neq 0$ implies that 
the function $t\mapsto \widetilde S(\pi+t\theta)$ (which is a quadratic polynomial) is strictly concave on $(-\delta,\delta)$ and cannot have a minimum at $t=0$, a contradiction.

\medskip

We have thus proved that for any minimizer of $S$ in $K$, we have $|\{y \in \Omega\,;\, \rho(y)\in (\alpha,1 - \alpha)\}|=0$ for all $\alpha>0$, that is $\rho(x)\in \{0,1\}$ a.e. in $\Omega$.


\section{Incompressible motion by chemotaxis with projection operator}\label{sec:projection}
In this section, we briefly recall some considerations presented in \cite{MRS} and relevant to this paper.
We consider the evolution of a density function $\rho(x,t)$ representing a population moving in response to the gradient of the potential $\phi$, subject to the incompressibility constraint $\rho(x,t)\leq 1$ (so no diffusion).
Since the continuity equation $\pa_t \rho + \div(\rho \na \phi)=0$ does not preserve the $L^\infty$ norm of $\rho$, 
we can consider the equation
\begin{equation}\label{eq:transport}
\begin{cases}
\pa_t \rho + \div(\rho P_{C(\rho)} (\na \phi))=0 & \mbox{ in } \Omega\times \R^+ \\[4pt]
\rho(x,0)= \rho_0(x) \leq 1 & \mbox{ in } \Omega.
\end{cases}
\end{equation}
where $P_{C(\rho)}$ denotes the projection operator onto the set of admissible velocity fields:
$$ C(\rho) = \left\{ v \in L^2(\Omega)^n\,;\, \int_\Omega v \cdot \na q \, dx \leq 0,\; \forall q\in H^1_\rho(\Omega)\right\} ,\quad H^1_\rho = \left\{ q\in H^1(\Omega) \, ;\, q\geq 0, \; q(1-\rho)=0 \mbox{ a.e.} \right\}
$$

Formally at least, this projection operator guarantees that the density satisfies the compressibility constraint 
since it implies $\div v\geq 0$ in (the interior of) the set $\{\rho=1\}$.
Note however that 
if $\{\rho=1\}$ is a set with empty interior (but positive Lebesgue measure), then we have $C(\rho) =  L^2(\Omega)^n$ and the constraint is not being enforced (see \cite{MRS}). Whether such a thing can happen when the initial condition is nice enough, is an interesting and probably challenging question.

The definition of $C(\rho)$ also impose $v\cdot n \leq 0$ on $\pa\Omega \cap \{\rho=1\}$ (nothing can leave the domain in the saturated region) which is crucial to the preservation of mass.
When $\phi$ satisfies Neumann conditions ($\alpha=0$) $\na \phi\cdot n=0$, the projection $P_{C(\rho)} (\na \phi)$ will satisfy the same conditions. When $\phi$ satisfy Robin conditions or Dirichlet conditions, we have $\na \phi\cdot n<0$ on $\pa\Omega$ and we have to impose $\rho=0$ on $\pa\Omega$ to have conservation of mass.
The condition $\rho v \cdot n =0$ in \eqref{eq:weak} includes both cases.

\medskip
As a reminder, the projection $v=P_{C(\rho)} (\na \phi)$ is uniquely defined since $C(\rho)$ is a convex set, and it is characterized by the following inequality:
\begin{equation}\label{eq:b}
\int_\Omega (v-\na \phi)\cdot (v-\xi) \, dx \leq 0, \qquad\forall \xi \in C(\rho).
\end{equation}

Classically, this projection, or minimization with constraint,  gives rise to a Lagrange multiplier in the form of a pressure term: One can show that there must exist $p\in H^1_\rho$ such that 
$$\na \phi - v= \na p.$$
The velocity $v =  P_{C(\rho)} (\na \phi)$ is in $C(\rho)$ and thus satisfies
\begin{equation} \label{eq:obst1}
\int_\Omega v \cdot \na q \, dx \leq 0,\quad \forall q\in H^1_\rho(\Omega)
\end{equation}
and the pressure satisfies the orthogonality condition (which follows from \eqref{eq:b} by taking $\xi=0$ and $\xi=2v$):
\begin{equation} \label{eq:obst2}
 \int_\Omega v \cdot \na p \, dx = 0.
\end{equation}
Using \eqref{eq:obst1} and \eqref{eq:obst2}, we get that $-\int v\cdot \na (p-q)\, dx \leq 0$ for all $q\in H^1_\rho$.
This implies that
$p$ solves the variational inequality (obstacle problem):
\begin{equation}\label{eq:obstacle}
\begin{cases}
p\in H^1_\rho \\
\displaystyle \int_\Omega[ \na p-\na \phi] \cdot \na (p-q) \leq 0 \, , \qquad \forall q\in H^1_\rho
\end{cases}
\end{equation}
which is our equation \eqref{eq:obst0}.
In fact, we can prove
\begin{proposition}
Assume $\phi \in H^1$ and 
let $p(x)$ be the (unique) solution of the variational inequality  \eqref{eq:obstacle}. Then 
$$ \na \phi - \na p  = P_{C(\rho)} (\na \phi).$$
\end{proposition}
\begin{proof}
Let $v=  \na \phi - \na p$. The variational inequality \eqref{eq:obstacle} gives
\begin{equation}\label{eq:ineq1} 
 - \int_\Omega v \cdot \na (p-q)\, dx \leq 0 \, , \qquad \forall q\in H^1_\rho.
 \end{equation}
For any $q_0\in H^1_\rho$, we have $q=p+q_0\in H^1_\rho$ and so this inequality implies
$$  \int_\Omega v \cdot \na q_0\, dx \leq 0 \, , \qquad \forall q_0\in H^1_\rho$$
that is $v\in C(\rho)$.
In particular we have $ \int_\Omega v \cdot \na p \, dx \leq 0$ and \eqref{eq:ineq1}  with $q=0$ gives $ \int_\Omega v \cdot \na p \, dx \geq 0$ so we must have the complementarity condition
$$ 
 \int_\Omega v \cdot \na p \, dx =0.$$
This complementarity condition can then be used to show that $v$ satisfies \eqref{eq:b}:
$$
\int_\Omega (v-\na \phi)\cdot (v-\xi) \, dx =  
-\int_\Omega \na p \cdot (v-\xi) \, dx 
=\int_\Omega \na p \cdot \xi \, dx 
\leq 0, \qquad\forall \xi \in C(\rho).
$$
by definition of $C(\rho)$ (since $p\in H^1_\rho$).
\end{proof}

\begin{remark}
We note that for $q\in H^1_\rho$ we also have $(1-\rho)\na q =0 $ a.e.
and so the solution of \eqref{eq:obstacle} also satisfies
$$ \int_\Omega \na p \cdot \na (p-q) - \rho \na \phi \cdot\na  (p-q) \leq 0 \, , \qquad \forall q\in H^1_\rho
$$
Proceeding as above, we can then show that $\rho \na \phi - \na p=\rho (\na \phi - \na p)  = P_{C(\rho)} (\rho \na \phi)$ that is
$$ \rho v = P_{C(\rho)} (\rho \na \phi).$$
\end{remark}

\begin{remark}\label{rmk:projk}
When we add diffusion to the transport equation, we need to consider $w=P_{C(\rho)}(- \mu \na  \log \rho + \na \phi )$, which is characterized by the inequality
$$
\int_\Omega (w-(- \mu \na  \log \rho +\na \phi))\cdot (w-\xi) \, dx \leq 0, \qquad\forall \xi \in C(\rho).
$$
Since $- \mu \na  \log \rho\in C(\rho)$ (using the fact that $\log \rho$ is maximum on $\{\rho=1\}$) we can take $\xi =\xi_0- \mu \na  \log \rho$ to get
$$
\int_\Omega (w+\mu \na  \log \rho - \na \phi)\cdot (w+\mu \na  \log \rho -\xi_0) \, dx \leq 0, \qquad\forall \xi _0\in C(\rho).
$$
This implies that $w+\mu \na  \log \rho =v$, that is
$
P_{C(\rho)}(- \mu \na  \log \rho +\na \phi )=  - \mu \na  \log \rho+ P_{C(\rho)}( \na \phi )
$.
\end{remark}


\bibliographystyle{plain}
\bibliography{JKO_chemotaxis}

\end{document}